\documentclass[11pt]{article}
\usepackage{amsthm,amsfonts,amsmath,amstext,amssymb,amsopn}
\usepackage[margin=1in]{geometry}
\usepackage{mathtools}
\usepackage[mathscr]{euscript}
\usepackage{setspace}
\usepackage{enumerate}
\usepackage{framed}
\usepackage{graphicx}
\graphicspath{{Figures/}}
\usepackage{caption}
\usepackage{subcaption}
\usepackage[breaklinks=true,colorlinks,citecolor=blue, linkcolor = black]{hyperref}
\usepackage{IEEEtrantools}
\usepackage{cases}
\usepackage[capitalise]{cleveref}
    \crefname{model}{Problem}{Problems}

\usepackage{pgf,tikz}
\usepackage{pgfplots}
\usetikzlibrary{intersections}
\usetikzlibrary{arrows}
\usetikzlibrary{positioning}
\tikzset{>=stealth',inner sep=0pt,outer sep=2pt}

\newtheorem{theorem}{Theorem}
\newtheorem{corollary}[theorem]{Corollary}

\newtheorem{lemma}[theorem]{Lemma}

\theoremstyle{definition}
\newtheorem{definition}{Definition}

\newtheorem{prop}{Proposition}

\usepackage{nicefrac}

\usepackage{mathtools}

\usepackage[backend=biber, style=numeric, url=false, doi=true]{biblatex}

\usepackage{hyperref}
\newcommand{\pderiv}[2]{\tfrac{\partial #1}{\partial #2}}
\newcommand\ztag[1]{%
\def\@currentlabel{#1}%
\gdef\tmp{%
\addtocounter{equation}{-1}%
\def\theequation{#1}}%
\aftergroup\aftergroup\aftergroup\aftergroup\aftergroup\aftergroup
\aftergroup\aftergroup\aftergroup\aftergroup\aftergroup\aftergroup
\aftergroup\aftergroup\aftergroup\aftergroup\aftergroup\aftergroup
\aftergroup\aftergroup\aftergroup\aftergroup\aftergroup\aftergroup
\aftergroup\aftergroup\aftergroup\aftergroup\aftergroup\aftergroup
\aftergroup
\tmp}

\makeatletter
\newcommand{\pushright}[1]{\ifmeasuring@#1\else\omit\hfill$\displaystyle#1$\fi\ignorespaces}
\newcommand{\pushleft}[1]{\ifmeasuring@#1\else\omit$\displaystyle#1$\hfill\fi\ignorespaces}
\makeatother

\usepackage[colorinlistoftodos]{todonotes}
\usepackage{etoolbox}

\newcommand\getcurrentref[1]{%
 \ifnumequal{\value{#1}}{0}
  {??}
  {\the\value{#1}}%
} 

\renewenvironment{framed}{%
}{%
    \ignorespacesafterend%
}

\let\Chaptermark\chaptermark
\def\chaptermark#1{\def\Chaptername{#1}\Chaptermark{#1}}
\let\Sectionmark\sectionmark
\def\sectionmark#1{\def\Sectionname{#1}\Sectionmark{#1}}
\let\Subsectionmark\subsectionmark
\def\subsectionmark#1{\def\Subsectionname{#1}\Subsectionmark{#1}}
\let\Subsubsectionmark\subsubsectionmark
\def\subsubsectionmark#1{\def\Subsubsectionname{#1}\Subsubsectionmark{#1}}

\definecolor{capri}{rgb}{0.0, 0.75, 1.0}
\definecolor{cherryblossompink}{rgb}{1.0, 0.72, 0.77}

\usepackage{lscape}

\newcommand{\R}{\mathbb R}

\newcommand{\brObj}{\mathscr{F}}

\newcommand{\caseZero}{\textbf{A}}
\newcommand{\caseZeroA}{\textbf{A1}}
\newcommand{\caseZeroB}{\textbf{A2}}
\newcommand{\caseZeroMinus}{\textbf{A$\boldsymbol{-}$}}

\newcommand{\caseOne}{\textbf{B}}

\newcommand{\caseTwoA}{\textbf{C1}}
\newcommand{\caseTwoB}{\textbf{C2}}
\newcommand{\caseTwoBMinus}{\textbf{C2$\boldsymbol{-}$}}
\newcommand{\caseTwoBPlus}{\textbf{C2$\boldsymbol{+}$}}
\newcommand{\caseTwoC}{\textbf{C3}}
\newcommand{\kspace}{ \hyperref[eq:k-space]{($k$-SNK)} }
\usepackage{authblk}

\begin{document}

\title{Continuous Equality Knapsack with Probit-Style Objectives\thanks{J. Fravel and R. Hildebrand were partially funded by ONR Grant N00014-20-1-2156.  Any opinions, findings, and conclusions or recommendations expressed in this material are those of the authors and do not necessarily reflect the views of the Office of Naval Research.  R. Hildebrand was also partially funded by the 2020-2021 Policy+ grant from Virginia Tech.}}
\author[1]{Jamie Fravel}
\author[1]{Robert Hildebrand\footnote{Corresponding author: \url{rhil@vt.edu}}}
\author[2]{Laurel Travis}

\affil[1]{Grado Department of Industrial and Systems Engineering, Virginia Tech, Blacksburg, Virginia, USA}
\affil[2]{Business Information Technology, Virginia Tech, Blacksburg, Virginia, USA}

\date{}
\maketitle

\begin{abstract}%
    We study continuous, equality knapsack problems with uniform separable, non-convex objective functions that are continuous, antisymmetric about a point, and have concave and convex regions.  For example, this model captures a simple allocation problem with the goal of optimizing an expected value where the objective is a sum of cumulative distribution functions of identically distributed normal distributions (i.e., a sum of inverse probit functions).  We prove structural results of this model under general assumptions and provide two algorithms for efficient optimization: (1) running in linear time and (2) running in a constant number of operations given preprocessing of the objective function.
\end{abstract}

\section{Introduction}%
We study the following restricted version of a nonlinear continuous knapsack problem 
\begin{definition}
    The Symmetric Nonlinear Continuous Knapsack problem is given by
   \begin{equation*} 
   \label{model:01space}
   \tag{SNK}
   \max \left\{ F(\mathbf{x}) = \sum_{i=1}^n f(x_i) :  \sum_{i=1}^n x_i = M , \mathbf{x} \in [0,1]^n \right\}
   \end{equation*}
where $f\colon [0,1]\rightarrow\mathbb{R}$ and $M \in [0,n]$.    
\end{definition}
The variables appear symmetrically in the problem, so any permutation of the variables for a solution yields an equivalent solution.  We make use of this property later in the paper via a reformulation of (\ref{model:01space}).

Our original interest in this model derives from a simple affine transformation to more general bounds on the variables.  In particular, consider the transformation $x_i \mapsto a + (b-a)x_i$ and define $f_0$ via $f_0(a + (b-a)x_i) = f(x)$ and $M_0$ via $M = \frac{M_0-an}{bn-an}$.  This results in a more general version of the problem which we call the $[a,b]$-Symmetric Nonlinear Continuous Knapsack problem and is given by 
   \begin{equation}
   \label{model:xspace}
   \tag{$[a,b]$-SNK}
   \max \left\{ F_0(\mathbf{x}) = \sum_{i=1}^n f_0(x_i) :  \sum_{i=1}^n x_i = M_0 , \mathbf{x} \in [a,b]^n \right\}
   \end{equation}
where $f_0\colon [a,b]\rightarrow\mathbb{R}$, and $a,b,M \in \R$ with $a \leq  b \leq M_0$. Any solution $\mathbf{x}$ of (\ref{model:01space}) can be transformed into a solution $\mathbf{x}'$ of (\ref{model:xspace}) by letting $x_i' = x_i(b-a)+a$. Therefore it suffices to investigate the solutions of (\ref{model:01space}) since all results will naturally extend to (\ref{model:xspace}). This focus will help keep our notation legible.

We make use of the following assumptions on $f$ for our main results: $f$ is antisymmetric about a point $c \in (0,1)$ (i.e. $f(x)-f(c) = f(c)- f\left(c-(x-c)\right)$ for all $x \in [0,1]$), and  $f$ is twice differentiable such that $f''(x) < 0$ for all $x \in (c,1)$.  By extension, we also have $f''(x) > 0$ for all $x \in (0,c)$. In particular, $f$ is strictly convex on $[0,c)$ and strictly concave on $(c,1]$.  See Figure \ref{fig:assumptions} for an example.

Several classes of functions satisfy these assumptions including logistic functions, probit models, tangent and arc-tangent functions, and the cumulative distribution functions (CDFs) of many probability distributions. Such sigmoidal and reverse-sigmoidal functions are used to model many phenomena with applications to marketing \cite{freeland1980s,AGRALI2009605}, target search \cite{zhang2013Newton}, competitive bidding \cite{rothkopf1969competetive}, and nature reserve design \cite{travis2008}. Most research focuses on strictly increasing functions, but our results extend to decreasing and even non-monotonic functions.

    \newcommand{\scaling}{0.3}
    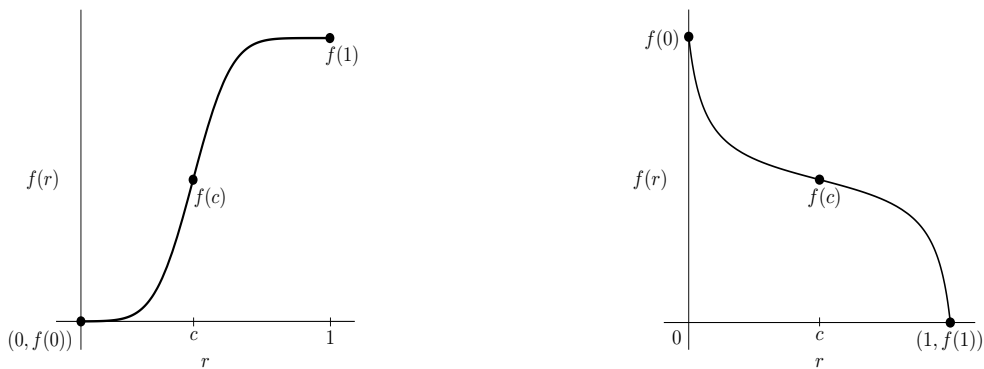
\begin{figure}[h!]
    \centering
        \begin{subfigure}{0.48\textwidth}
        \centering
            \resizebox{0.6\textwidth}{0.6\textwidth}{%
            \begin{tikzpicture}
            \node (img)  {\includegraphics[scale=0.5]{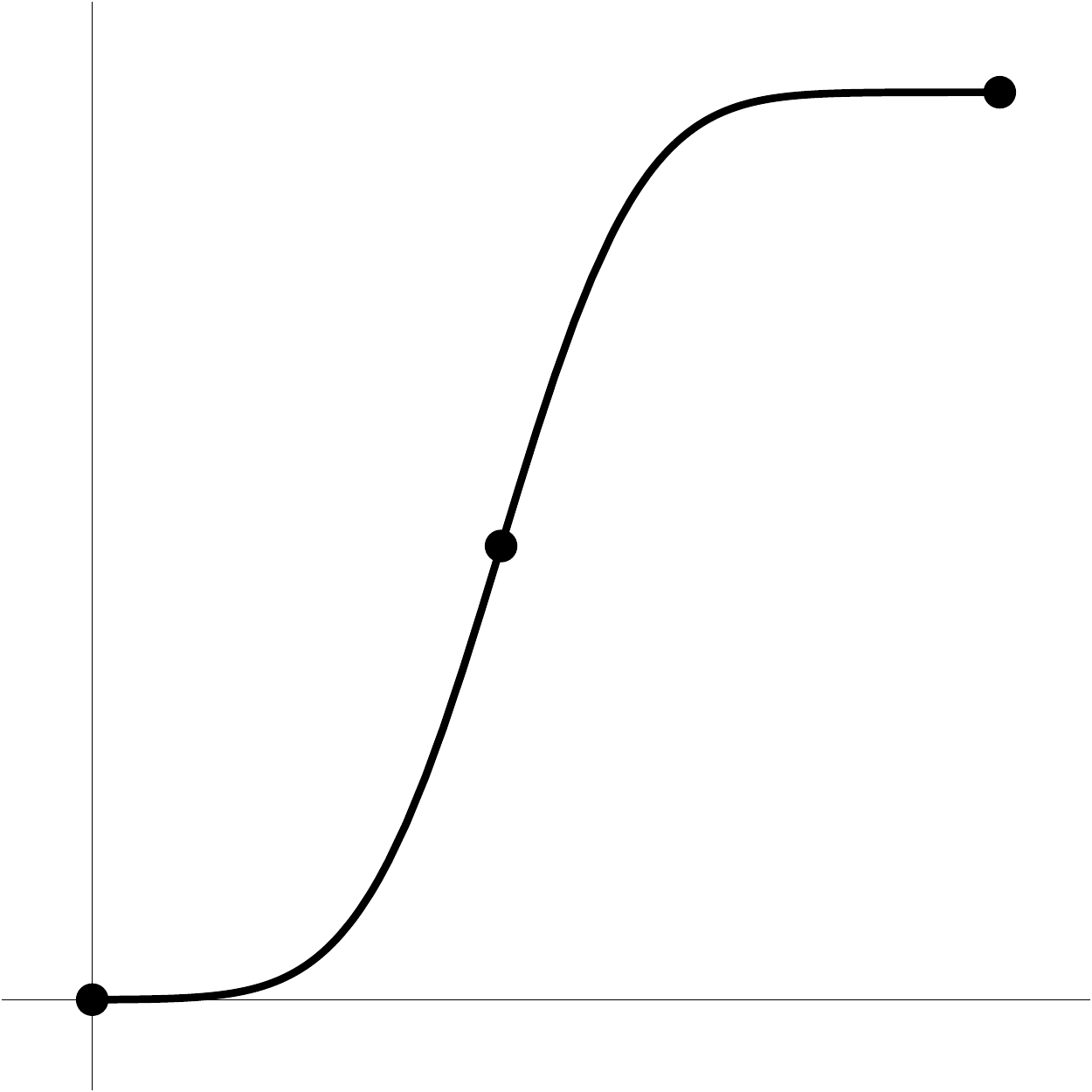}};
            \node[below=of img, node distance=0cm, yshift=1cm] {$r$};
            \node[left=of img, node distance=0cm, anchor=center,xshift=0.8cm] {$f(r)$};
            \node[below=of img, node distance=1cm, yshift=1.5cm,xshift=-3.5cm] {$(0,f(0))$};
            \node[above=of img, node distance=1cm, yshift=-2.2cm,xshift=2.9cm] {$f(1)$};
            \node[below=of img, node distance=1cm, yshift=4.17cm,xshift=0.1cm] {$f(c)$};
            \node[below=of img, node distance=1cm, yshift=1.5cm,xshift=2.65cm] {$1$};
            \node[below=of img, node distance=1cm, yshift=1.75cm,xshift=2.65cm] {$\rule{.4pt}{1ex}$};
            \node[below=of img, node distance=1cm, yshift=1.5cm,xshift=-0.25cm] {$c$};
            \node[below=of img, node distance=1cm, yshift=1.75cm,xshift=-0.25cm] {$\rule{.4pt}{1ex}$};
            \end{tikzpicture}
        }
        \caption{A valid function based on the normal CDF.}
        \end{subfigure}
        ~
        \begin{subfigure}{0.48\textwidth}
        \centering
        \resizebox{0.6\textwidth}{0.6\textwidth}{%
            \begin{tikzpicture}
            \node (img)  {\includegraphics[scale=0.39]{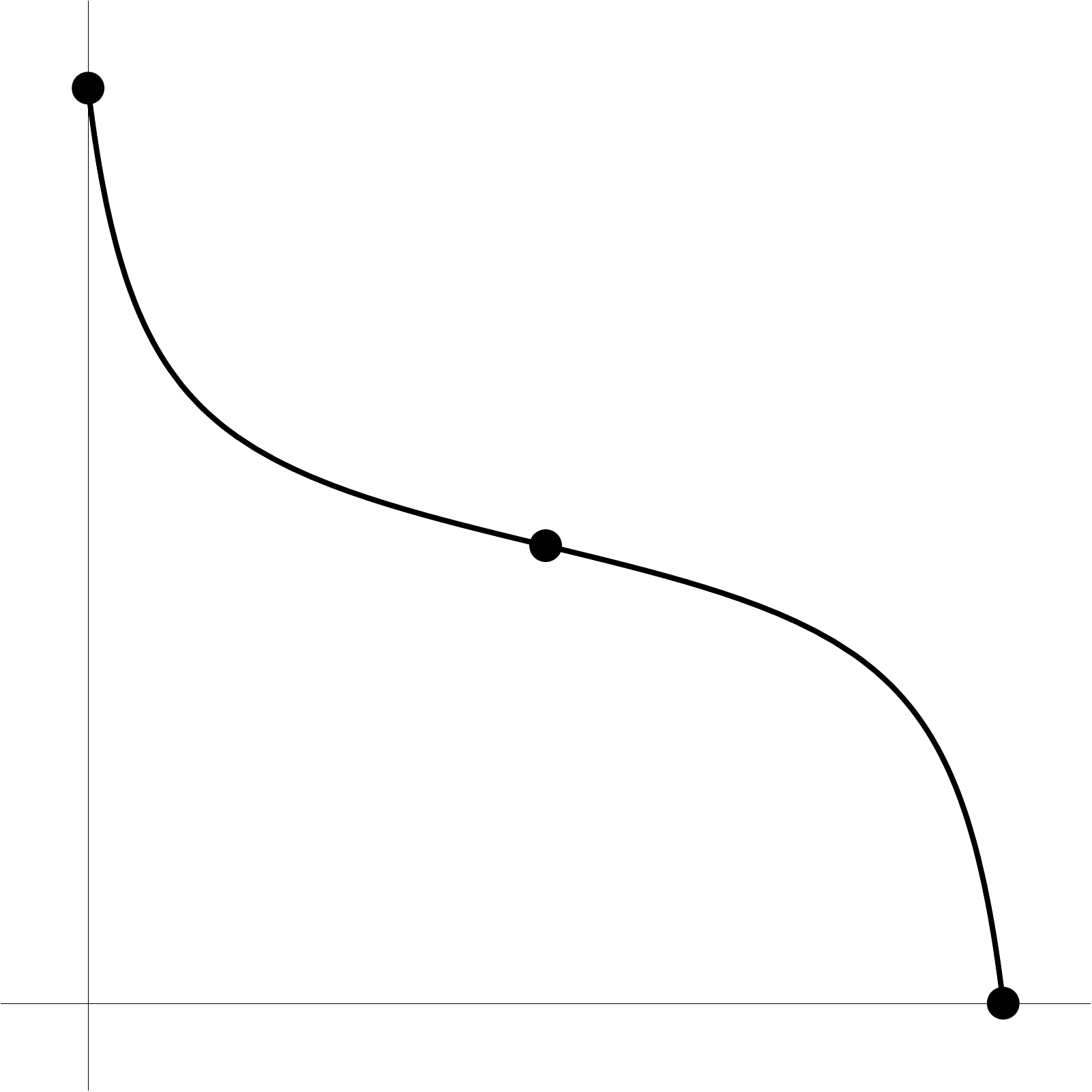}};
            \node[below=of img, node distance=0cm, yshift=1cm] {$r$};
            \node[left=of img, node distance=0cm, anchor=center,xshift=0.8cm] {$f(r)$};
            \node[below=of img, node distance=1cm, yshift=1.5cm,xshift=-2.9cm] {$0$};
            \node[above=of img, node distance=1cm, yshift=-1.9cm,xshift=-3.2cm] {$f(0)$};
            \node[below=of img, node distance=1cm, yshift=4.17cm,xshift=0.1cm] {$f(c)$};
            \node[below=of img, node distance=1cm, yshift=1.5cm,xshift=2.65cm] {$(1,f(1))$};
            \node[below=of img, node distance=1cm, yshift=1.75cm,xshift=2.65cm] {$\rule{.4pt}{1ex}$};
            \node[below=of img, node distance=1cm, yshift=1.5cm,xshift=0cm] {$c$};
            \node[below=of img, node distance=1cm, yshift=1.75cm,xshift=0cm] {$\rule{.4pt}{1ex}$};
            \end{tikzpicture}
        }
        \caption{A valid function based on negative tangent.}
        \end{subfigure}
    \caption{Two example of functions that are twice-differentiable, antisymmetric about $c$, and strictly convex on $[0,c)$.}
    \label{fig:assumptions}
    \end{figure}

These results were inspired by analysis of racial representation in legislative redistricting in the United States by using     \begin{equation*}\label{obj:BlackReps}
    f(r) := \Phi\left(\beta\cdot r-\beta_0\right),
    \end{equation*}
where $\Phi$ is the cumulative distribution function of the standard normal distribution and $r$ is the ratio of black voters to the total population of voters in a given district. After calibrating $\beta$ and $\beta_0$ from historical data, this function predicts the probability that a district will elect a black congressional representative. More details on this work can be seen in \cite{bvap_opt, bvap_southern, political}\footnote{
Although we do not use it in this article, data related to these redistricting problems can be provided if requested.
}.

\paragraph{Prior work}

Knapsack and resource allocation problems are well-studied in the literature, particularly for linear objectives. We mention just a few references that are most related to our work. Perhaps the earliest reference with nonlinear objectives is Derman~\cite{Derman1959}  in 1959.  Derman studies a similar variant to our allocation problem and then handles allocation under random demand.  Derman does not require separability of the objective function.
A similar setup was used under an environmental application in~\cite{travis2008} to identify a starting point for a simulation optimization heuristic with the objective of maximizing the probability that an endangered species persists.
Kodilam and Luss~\cite{Kodialam1998} study the allocation problem while minimizing a separable convex objective.  
Ağralı and Geunes \cite{AGRALI2009605} provide a polynomial time approximation algorithm for knapsack with separable, sigmoid objectives and a similar polynomial exact algorithm in the case of identical separability without assuming antisymmetry. 
Srivastava and Bullo \cite{srivastava2014knapsack} explore a few problems, including the knapsack problem, with sigmoid objectives and incidentally arrive at an $O(n)$ algorithm for solving \eqref{model:01space}.
Hochbaum~\cite{HOCHBAUM1995103} studies a nonlinear integer or continuous knapsack maximization problem with a separable concave objective and a number of variants, providing efficient approximation algorithms for these problems. 
These are perhaps the most clear complexity results for this problem class.
Bitran and Hax~\cite{10.2307/2631335} study the resource allocation problem with a separable convex objective and suggest a recursive procedure to solve these problems.  Bretthaur and Shetty~\cite{10.2307/171693,BRETTHAUER2002505} study a generalization of the nonlinear resource allocation problem with nonlinear constraints.  They present an algorithm based on branch and bound.
For more results and applications of general nonlinear knapsack, see~\cite{PATRIKSSON20081,patriksson2015algorithms} for surveys of the continuous nonlinear resource allocation problem.

\paragraph{Contributions and outline}
Our main contribution (Theorem~\ref{thm:opt_int_sol}) is to show that (\ref{model:01space}) can be solved in a constant number of operations, provided some numerical preprocessing (see Proposition~\ref{prop:Compute_d_0}) is done on the function $f$. We prove these results using only the basic assumptions discussed after (\ref{model:01space}).  We prove a number of structural results along the way, which also imply a simple polynomial time algorithm (Corollary~\ref{cor:poly-time}) to solve (\ref{model:01space}).  

We begin with a KKT analysis of the optimal solutions that leads to the simple linear time algorithm.  We then delve deeper into the structure of optimal solutions through a reformulation of the problem into a MINLP.  Analysis of this problem reveals a finite number of potential optimal solutions to consider, which can be computed given proper preprocessing of the function.  In particular, we need to compute a domain value $d_0$ with specific properties related to $f$.

\section{Structure of KKT Solutions}%
 
Unless stated otherwise, we assume that $M \in [0,n]$.
Note that (\ref{model:01space}) is feasible if and only if $M \in [0,n]$.
Given a feasible solution $\mathbf{x}$ of (\ref{model:01space}), partition $I = \{1,\dots,n\}$ into 
    $$
    I_\mathbf{x}^0 = \{i \in I\ :\ x_i = 0\},\qquad I_\mathbf{x}^1 = \{i \in I\ :\ x_i = 1\}, \quad \text{and}\quad I_\mathbf{x}^y = \{i \in I\ :\ x_i \in (0,1)\}.
    $$
\begin{lemma}\label{lem:lambda}
   Suppose that $f$ is differentiable on $[0,1]$. Let $\mathbf{x}^*$ be an optimal solution of (\ref{model:01space}). There exists a constant $\lambda$ such that
        \begin{enumerate}[(a)]
        \item $f'(x_i^*) = \lambda$ for every $i$ in $I_\mathbf{x*}^y$,              \label{lem:lambda:y}
        \item If $I_{\mathbf{x}^*}^0$ is nonempty, then $f'(0) \leq \lambda$, and   \label{lem:lambda:a}
        \item If $I_{\mathbf{x}^*}^1$ is nonempty, then $\lambda \leq f'(1)$.       \label{lem:lambda:b}
        \end{enumerate}
\end{lemma}
This lemma uses KKT conditions to analyze optimal solutions.  This is a common technique used in the literature applied to variants of this problem. 
 See, for instance, Patriksson and Str\"omberg \cite[Equations (3c-3e)]{patriksson2015algorithms} for a similar derivation on a more general version of the problem. 
\begin{framed}
	\begin{proof}
	    The Lagrangian of (\ref{model:01space}) is given by 
	        \begin{equation*}
			L(\mathbf{x},\boldsymbol{\mu}^0, \boldsymbol{\mu}^1, \lambda) = \sum_{i=1}^n f(x_i) + \sum_{i=1}^n x_i \mu_i^0 + \sum_{i=1}^n (1-x_i) \mu_i^1 + \lambda \left(M-\sum_{i=1}^n x_i\right).
		\end{equation*}
	    Notice that 
	        \begin{equation*}
	        \pderiv{}{x_i}L(\mathbf{x},\boldsymbol{\mu}^0, \boldsymbol{\mu}^1, \lambda) = f'(x_i)+\mu_i^0-\mu_i^1-\lambda.
	        \end{equation*}
	   Since $\mathbf{x}^*$ is optimal, the KKT necessary conditions tell us that, for each $i \in \{1,...,n\}$,
	        \begin{subnumcases}{}
	        f'(x_i^*) = \lambda-\mu_i^0+\mu_i^1   \\
	        x_i^* \mu_i^0 = 0                \label{prf:CompSlacka}\\
	        (1-x_i^*) \mu_i^1 = 0                \label{prf:CompSlackb}\\
	        \boldsymbol{\mu}^0,\boldsymbol{\mu}^1 \geq \boldsymbol{0}.
	        \end{subnumcases}
	    If $i \in I_{\mathbf{x}^*}^y \cup I_{\mathbf{x}^*}^1$, then $x_i^* > 0$ so \eqref{prf:CompSlacka} implies that $\mu_i^0 = 0$. Similarly, if $i \in I_{\mathbf{x}^*}^y \cup I_{\mathbf{x}^*}^0$, then $1-x_i^* > 0$ so \eqref{prf:CompSlackb} implies that $\mu_i^1 = 0$. These conditions reduce to 
	       \begin{subnumcases}{}
            f'(x_i^*)=\lambda,      &$\forall\ i\in I_{\mathbf{x}^*}^y \label{prf:KKTy}$\\
            f'(0)=\lambda-\mu_i^0,  &$\forall\ i\in I_{\mathbf{x}^*}^0 \label{prf:KKTa}$\\
            f'(1)=\lambda+\mu_i^1,  &$\forall\ i\in I_{\mathbf{x}^*}^1 \label{prf:KKTb}$\\
            \boldsymbol{\mu}^0,\boldsymbol{\mu}^1 \geq \boldsymbol{0}.
	        \end{subnumcases}
	    Notice that \eqref{prf:KKTy} is exactly condition (a). Since $\mu_i^0 \geq 0$, \eqref{prf:KKTa} can only be satisfied if $f'(0) \leq \lambda$; similarly, \eqref{prf:KKTb} can only be satisfied if $f'(1) \geq \lambda$.  Thus, we have (b) and (c). Note that \eqref{prf:KKTy}, \eqref{prf:KKTa}, and \eqref{prf:KKTb} only apply if their respective sets are nonempty.
	\end{proof}
\end{framed}

We call a solution $\mathbf{x}$ \textit{trine} if each entry $x_i$ takes one of three possible values. In particular, we use trine to describe solutions for which $x_i \in \{0,y,1\}$ for every $i$ and some $y \in (0,1)$. That is, $x_i = y$ for every $i \in I_\mathbf{x}^y$. Additionally, we define the antisymmetric complement $\overline{x} = 2c-x$ of a point $x\in[0,1]$ so that $f'(x) = f'(\overline{x})$ by the antisymmetry of $f$.

\begin{lemma}\label{lem:trine}
    Suppose that $f$ is differentiable, antisymmetric about $c$, and strictly convex on $[0,c)$. Then there is an optimal trine solution $\mathbf{x}^*$ of (\ref{model:01space}). 
\end{lemma}

Similar results with different assumptions are known, for instance, in solving knapsack problems with concave objectives (e.g., \cite[Equation 5]{patriksson2015algorithms}). By assuming antisymmetry, we can achieve this useful result for our non-concave objectives.

\begin{framed}
	\begin{proof}
        It may be that $c$ is not the midpoint of $0$ and $1$; we will focus on the case where $1-c > c$ as the other cases follow by a similar argument. 
		
		In this case $\overline{0} = 2c < 1$. By the antisymetry of $f$, we have $f(x) = 2f(c)-f(\overline{x})$ and thus
		    $$
			f(x)+f(\overline{x}) \quad=\quad 2f(c)-f(\overline{x})+f(\overline{x}) \quad=\quad 2f(c)\quad
		    $$
		for any $x \in [c,1]$.
		
		Let $\mathbf{x}^*$ be an optimal, non-trine solution of (\ref{model:01space}); naturally, it must satisfy Lemma \ref{lem:lambda}. By Lemma \ref{lem:lambda}(\ref{lem:lambda:y}) it must be that $f'(x_i^*)$ takes the same value for each $i \in I_{\mathbf{x}^*}^y$; call this value $\lambda$. Notice that since $f'$ is strictly decreasing over $(c,1]$ (due to strict convexity on $[0,c)$ and antisymmetry) there is at most one point $y$ in $[c,1)$ which satisfies $f'(y) = \lambda$. However, by antisymmetry, $f'(\overline{y}) = f'(y) = \lambda$. Since $y \in [c,1)$, we know that $\overline{y} \leq c$.
		
		Since $\mathbf{x}^*$ is not trine but does satisfy Lemma \ref{lem:lambda}, there exists a pair of distinct indices $j,k$ such that $x_j^* = y$ and $x_k^* = \overline{y}$. Consider the perturbed solution $\mathbf{x}'$ defined by
		    $$
		    x'_i = \begin{cases}
    			y+\Delta, & \textrm{if } i = j \\
    			\overline{y}-\Delta, & \textrm{if } i = k \\
    			x_i, &\textrm{otherwise}
    		    \end{cases}
		    $$
		for some $\Delta \in [\overline{y}-c,\overline{y}]\setminus\{0\}$. Notice that $\overline{y}-\Delta = 2c-(y+\Delta) = \overline{y+\Delta}$ and thus
		    \begin{equation*}
			F(\mathbf{x})-F(\mathbf{x}') \quad=\quad f(y)+f(\overline{y})-f(y+\Delta)-f(\overline{y}-\Delta) \quad=\quad 2f(c)-2f(c) \quad=\quad 0.
		    \end{equation*}
		That is, we may perturb antisymmetric pairs in a solution without changing the objective function value. Further, such a permutation does not change $\sum_{i=1}^nx_i$ so the constraint $\sum_{i=1}^n x_i = M$ 
		will not be violated and the bounds on $\Delta$ are chosen to prevent violation of $\mathbf{x} \in [0,1]^n$. 
		
        If $x_i \in \{0,1\}$ for all $i \notin \{j,k\}$, then such a perturbation does not change the satisfaction of Lemma \ref{lem:lambda}.\ref{lem:lambda:y} and $\mathbf{x}'$ is also optimal. Letting $\Delta = \overline{y}$ gives $x_j' = 2c = \overline{0}$ and $x_k' = 0$ so, in this case, $x_i' \in\{0,2c,1\}$ for each $i \in \{1,\dots,n\}$ so $\mathbf{x}'$ is trine. That is, we may perturb any non-trine, optimal solution with $\vert I_{\mathbf{x}^*}^y \vert = 2$ into a trine, optimal solution.
		
        On the other hand, if there exists any $i \notin \{j,k\}$ for which $x_i \in\{y,\overline{y}\}$ then applying the perturbation for any $\Delta$ means that the entries of $\mathbf{x}'$ span either $\{0,\overline{y},\overline{y}-\Delta,y+\Delta,1\}$ or $\{0,\overline{y}-\Delta,y+\Delta,y,1\}$. In either case, Lemma \ref{lem:lambda}(\ref{lem:lambda:y}) is violated and $\mathbf{x}'$ is not optimal, but $F(\mathbf{x}') = F(\mathbf{x})$ which contradicts our assumption that $\mathbf{x}$ is optimal. That is, no non-trine solution with
		$\vert I_{\mathbf{x}^*}^y \vert \geq 3$ is optimal.
	\end{proof}
\end{framed}

Define $\mathcal{X}(k_0,k_1)$ to be the family of trine solutions $\mathbf{x}$ of (\ref{model:01space}) such that $\vert I_\mathbf{x}^0 \vert = k_0$ and $\vert I_\mathbf{x}^1 \vert = k_1$. That is, $k_0$ many elements of $\mathbf{x}$ take the value $a$ and $k_1$ many elements take the value $b$.

Define $k_y = n-k_0-k_1$ so that $\vert I_\mathbf{x}^y \vert = k_y$. Since each $\mathbf{x} \in \mathcal{X}(k_0,k_1)$ is trine, the constraint $\sum_{i=1}^n x_i = M$ 
tells us that
    \begin{equation}
    \begin{array}{ccccc}
    k_1+yk_y = M &\Rightarrow& y = \frac{M-k_1}{n-k_0-k_1}.\label{eq:yInTermsOfky}
    \end{array}
    \end{equation}
Thus each family defines a unique $k_y$ and, assuming $k_y > 0$, a unique $y$.

\begin{corollary}[Enumeration algorithm]
\label{cor:poly-time}
Suppose that $f$ is differentiable, antisymmetric about $c$, and strictly convex on $[0,c)$. There exists an algorithm that runs in $O(n^2)$ time and evaluations of $f$ that computes the optimal solution of (\ref{model:01space}).
\end{corollary}
\begin{proof}
    Enumerate partitions of $n$  into three integers, $k_0$, $k_1$, and $k_y$, of which there are $O(n^2)$ many such choices.   Compute $y$ via \eqref{eq:yInTermsOfky} (when $k_y \neq 0$).  Then compare all values of the function $f(0)k_0+f(1)k_1+f(y)k_y$ and return the largest.
\end{proof}

Note that this algorithm can be improved to $O(n)$ time complexity when you can argue that either $k_0$ or $k_1$ is zero in an optimal solution as was accomplished by Srivastava and Bullo \cite[Corollary 5]{srivastava2014knapsack} for identically-separable knapsack problems with sigmoid objectives.

\section{A $k$-space model}
We can express the objective function value of a solution in terms of our new parameters:
	$$
	F(\mathbf{x}) = f(0)k_0+f(1)k_1+f(y)k_y,
	$$
for any $\mathbf{x} \in \mathcal{X}(k_0,k_1)$. 
Therefore, solutions are homogeneous within their respective groups and we can construct a simplified model that takes advantage of this structure.

\begin{definition}
\label{model:kspace}
The $k$-space reformulation of (\ref{model:01space}) is given by  

 \begin{minipage}{0.05\textwidth}
    \begin{flushleft}
        \text{($k$-SNK)}
    \end{flushleft}
\end{minipage}
\begin{minipage}{0.90\textwidth}
    \begin{subequations}
    \label{eq:k-space}
	\begin{alignat}{2}
		\textrm{Maximize}& \quad &  \mathclap{\hspace{40pt}\brObj(\mathbf{k}) = f(0)k_0+f(1)k_1+f(y)k_y} 	  \label{obj:kspace}\\
			\textrm{s.t.}& \quad &     k_0+k_1+k_y &= n   \label{constr:kspacen}\\
                             & \quad &  k_1+yk_y &= M 	     \label{constr:kspaceM}\\
			 			 & \quad &     k_0,k_1,k_y &\in \mathbb{Z}_+                   \label{constr:kspaceint}\\
			 			 & \quad &               y &\in (0,1)                        \label{constr:kspaceynds}
	\end{alignat}
	\end{subequations}
\end{minipage}

\end{definition}
A solution of \kspace takes the form $(k_0,k_1,k_y,y)$, is uniquely defined by any two of its elements, and represents a family of trine solutions of (\ref{model:01space}) (if $k_y = 0$, $y$ does not contribute to the satisfaction of \eqref{constr:kspaceM} or to the objective function and so can take any value). Assuming that $f$ is differentiable, antisymmetric about $c$, and strictly convex on $[0,c)$; \kspace has the same optimal value as (\ref{model:01space}). We proceed by solving the continuous relaxation of this MINLP.

\subsection{Continuous Relaxation}%

We call a solution $\mathbf k = (k_0,k_1,k_y,y)$ \textit{equality-feasible} if it satisfies the equality constraints \eqref{constr:kspacen} and \eqref{constr:kspaceM} of \kspace, but not necessarily non-negativity and integrality \eqref{constr:kspaceint} or the bounds on $y$ \eqref{constr:kspaceynds}.

\begin{lemma}[Feasibility in terms of $k_0$ and $k_1$]
\label{lem:feasibility}
    An equality-feasible solution $(k_0,k_1,k_y,y)$ is fully feasible to the continuous relaxation of \kspace if and only if
        \begin{equation*}
        \max\left(\tfrac{yn-M}{y},0\right) \quad\leq\quad k_0 \quad\leq\quad n-M     \qquad\textrm{and}\qquad  \max\left(\tfrac{M-yn}{1-y},0\right) \quad\leq\quad k_1 \quad\leq\quad M.
        \end{equation*}
\end{lemma}
\begin{framed}
    \begin{proof}
 Constraints \eqref{constr:kspacen} and \eqref{constr:kspaceint}  imply 
$
            0 \leq k_0,k_1,k_y \leq n %
$.  Using the equations $k_1+yk_y = M$ and $k_0+k_1+k_y = n$ from \kspace, the bounds on each $k$ are easily derivable by projecting out the other variables.

    \end{proof}
\end{framed}

\begin{lemma}\label{lem:dr}
    Suppose that $f$ is twice-differentiable, antisymmetric about $c$, and strictly convex on $[0,c)$. For any $r \in [0,1]\setminus \{c\}$, the equation
        \begin{equation}
        f(r) = f(x)+(r-x)f'(x) \label{def:dr}
        \end{equation}
    has exactly two solutions: $x=r$ and another, which lies between $\overline{r} = 2c-r$ and $c$. We denote this additional solution $d_r$.
\end{lemma}
Define $d_0$ and $d_1$ as the $d_r$ solutions to equation \eqref{def:dr} for $r=0$ and $r=1$, respectively. See Figure \ref{fig:dr} for a graphical definition of these two points. The point $d_0$ can be used to define the concave envelope of the objective as was done by Freeland and Weinberg in \cite{freeland1980s}, but approximation performance of this envelope is arbitrarily bad \cite{srivastava2014knapsack}. Instead, we define $d_0$ because of its influence on the optimal solution of \kspace. 

\renewcommand{\scaling}{0.6}
    \begin{figure}[h!]\centering
    \begin{subfigure}{0.49\textwidth}
        \centering
        \resizebox{\scaling\textwidth}{\scaling\textwidth}{%
            \begin{tikzpicture}
            \node (img)  {\includegraphics[scale=0.5]{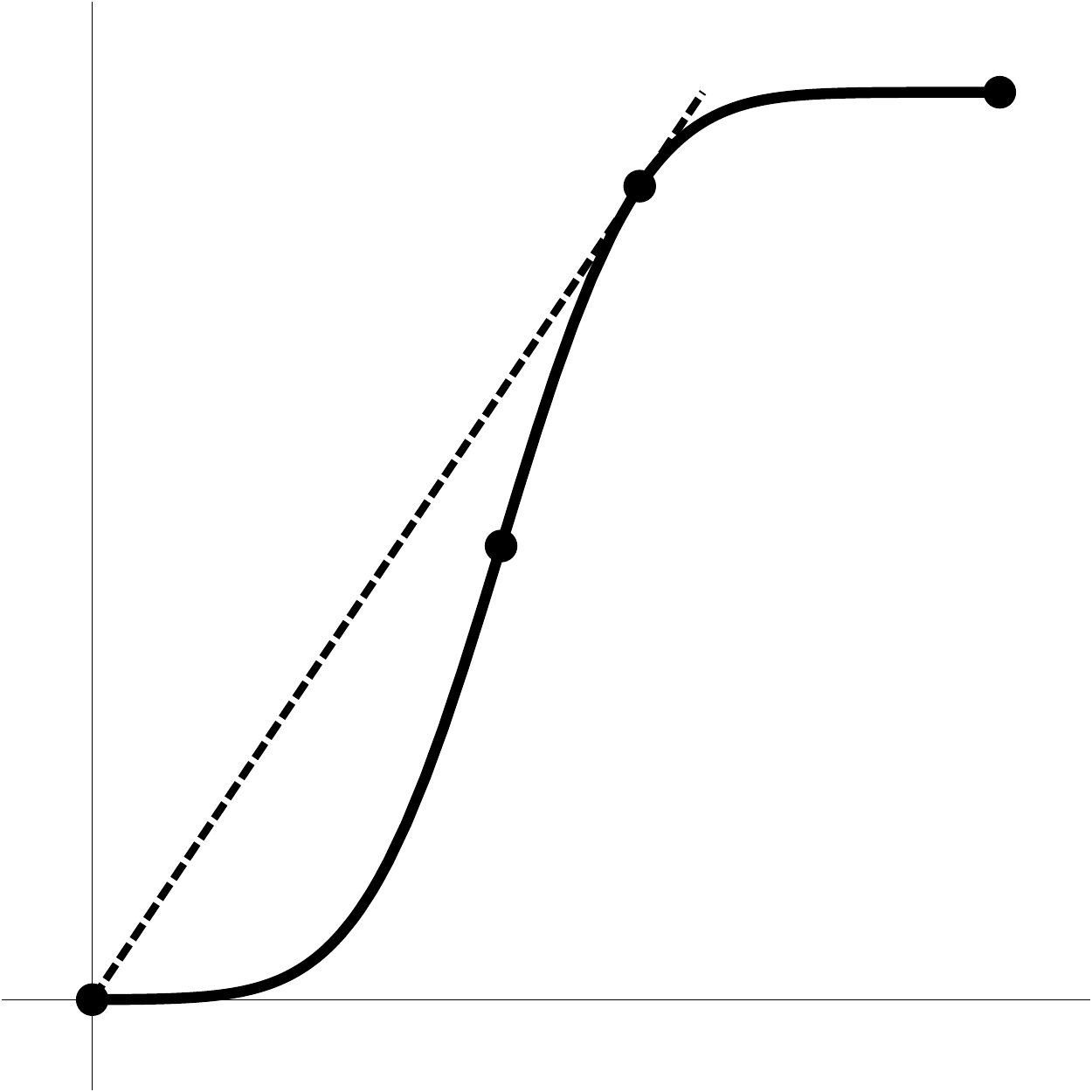}};
            \node[below=of img, node distance=1cm, yshift=1.5cm,xshift=-3.5cm] {$(0,f(0))$};
            \node[above=of img, node distance=1cm, yshift=-2.9cm,xshift=1.3cm] {\LARGE$f(d_0)$};
            \node[below=of img, node distance=1cm, yshift=1.5cm,xshift=0.65cm] {\LARGE$d_0$};
            \node[below=of img, node distance=1cm, yshift=1.9cm,xshift=0.5cm] {$\rule{0.4pt}{2ex}$};
            \node[left=of img, node distance=0cm, anchor=center,xshift=0.8cm] {$f(r)$};
            \node[below=of img, node distance=1cm, yshift=1.5cm,xshift=2.65cm] {$1$};
            \node[below=of img, node distance=1cm, yshift=1.75cm,xshift=2.65cm] {$\rule{.4pt}{1ex}$};
            \node[below=of img, node distance=1cm, yshift=1.5cm,xshift=-0.25cm] {$c$};
            \node[below=of img, node distance=1cm, yshift=1.75cm,xshift=-0.25cm] {$\rule{.4pt}{1ex}$};
            \node[below=of img, node distance=0cm, yshift=0.8cm] {$r$};
            \end{tikzpicture}
        }
        \caption{Tangent definition of $d_0$}
        \label{fig:da}
    \end{subfigure}
    \begin{subfigure}{0.49\textwidth}
        \centering
        \resizebox{\scaling\textwidth}{\scaling\textwidth}{%
            \begin{tikzpicture}
            \node (img)  {\includegraphics[scale=0.5]{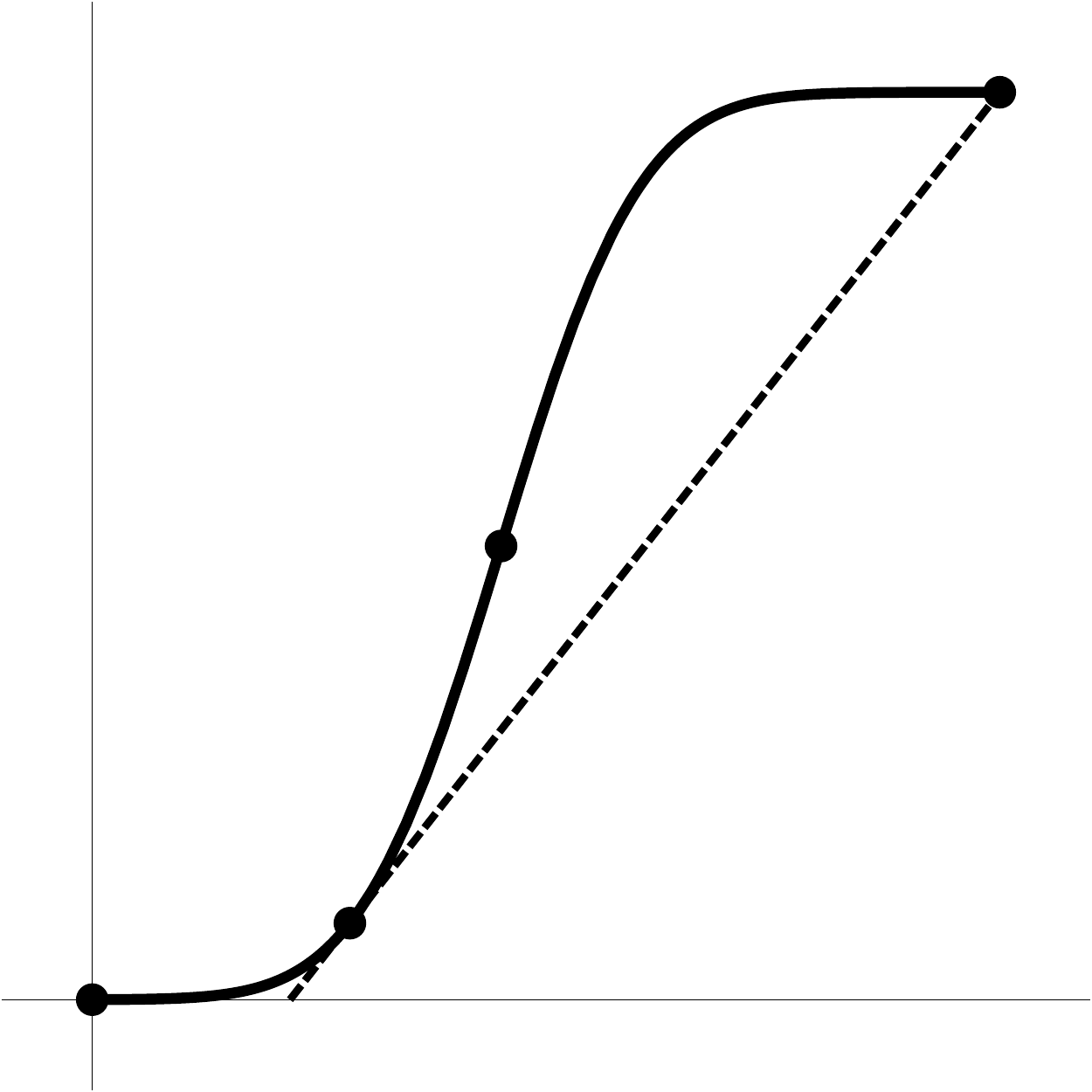}};
            \node[below=of img, node distance=1cm, yshift=1.5cm,xshift=-2.9cm] {$0$};
            \node[below=of img, node distance=1cm, yshift=2.7cm,xshift=-1.9cm] {\LARGE$f(d_1)$};
            \node[below=of img, node distance=1cm, yshift=1.5cm,xshift=-1.03cm] {\LARGE$d_1$};
            \node[below=of img, node distance=1cm, yshift=1.9cm,xshift=-1.13cm] {$\rule{0.4pt}{2ex}$};
            \node[below=of img, node distance=0cm, yshift=1cm] {$r$};
            \node[left=of img, node distance=0cm, anchor=center,xshift=0.8cm] {$f(r)$};
            \node[below=of img, node distance=1cm, yshift=1.5cm,xshift=2.65cm] {$1$};
            \node[below=of img, node distance=1cm, yshift=1.75cm,xshift=2.65cm] {$\rule{.4pt}{1ex}$};
            \node[below=of img, node distance=1cm, yshift=1.5cm,xshift=-0.25cm] {$c$};
            \node[below=of img, node distance=1cm, yshift=1.75cm,xshift=-0.25cm] {$\rule{.4pt}{1ex}$};
            \node[above=of img, node distance=1cm, yshift=-2.2cm,xshift=2.9cm] {$f(1)$};
            \end{tikzpicture}
        }
        \caption{Tangent definition of $d_1$}
        \label{fig:gx}
    \end{subfigure}
    \caption{ The value $d_r$ defined  such that the tangent line of $f$ at $(d_r, f(d_r))$  passes through $(r,f(r))$.  We highlight important examples in this figure.}
    \label{fig:dr}
    \end{figure}
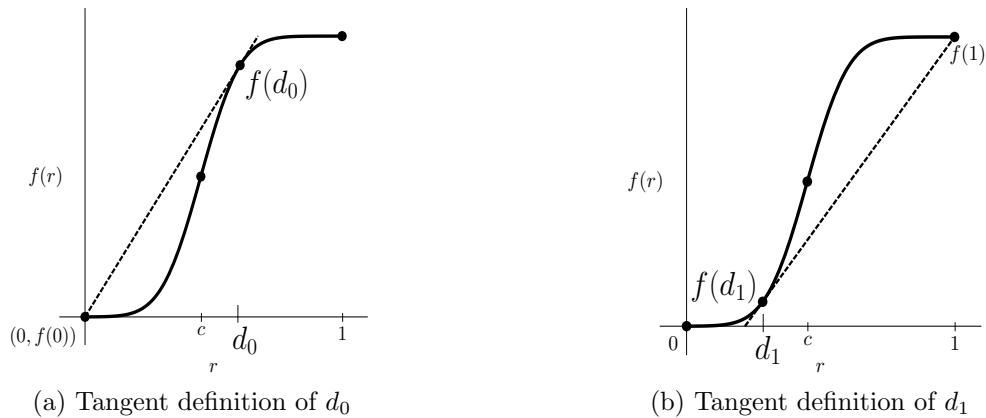

\begin{framed}
    \begin{proof}
        Because $f$ is antisymmetric about $c$, it suffices to show these properties for some $r < c$. 
        
        Define $g(x) := f(x)+f'(x)(r-x)-f(r)$. Clearly $g(r) = 0$, so $x = r$ is a solution of equation~\eqref{def:dr}. We claim that $g(c) < 0$ and $g(2c-r) > 0$.   Thus, by the Intermediate Value Theorem, we conclude that there exists a point $d_r \in (c,2c-r)$ such that $g(d_r) = 0$.
        
        First note that since $f(x)$ is strictly convex over $[0,c)$, the first-order-convexity of $f$ at $x = c$ implies $g(c) = f(c) + f'(c)(r-c)-f(r) < 0$. Now notice, by the antisymmetry of $f$, that
            \begin{align*}
            g(2c-r) &= f(2c-r)+f'(2c-r)\left(r-(2c-r)\right)-f(r) = 2f(c)+2f'(r)(r-c)-2f(r) \\
                &= -2\left(f(r)+f'(r)(c-r)-f(c)\right) > 0,
            \end{align*}
        where the last inequality follows from  first-order-convexity of $f$ applied at $x=r$.
        
        \begin{figure}
        \centering
            \begin{tikzpicture}
            \node (img)  {\includegraphics[scale=0.5]{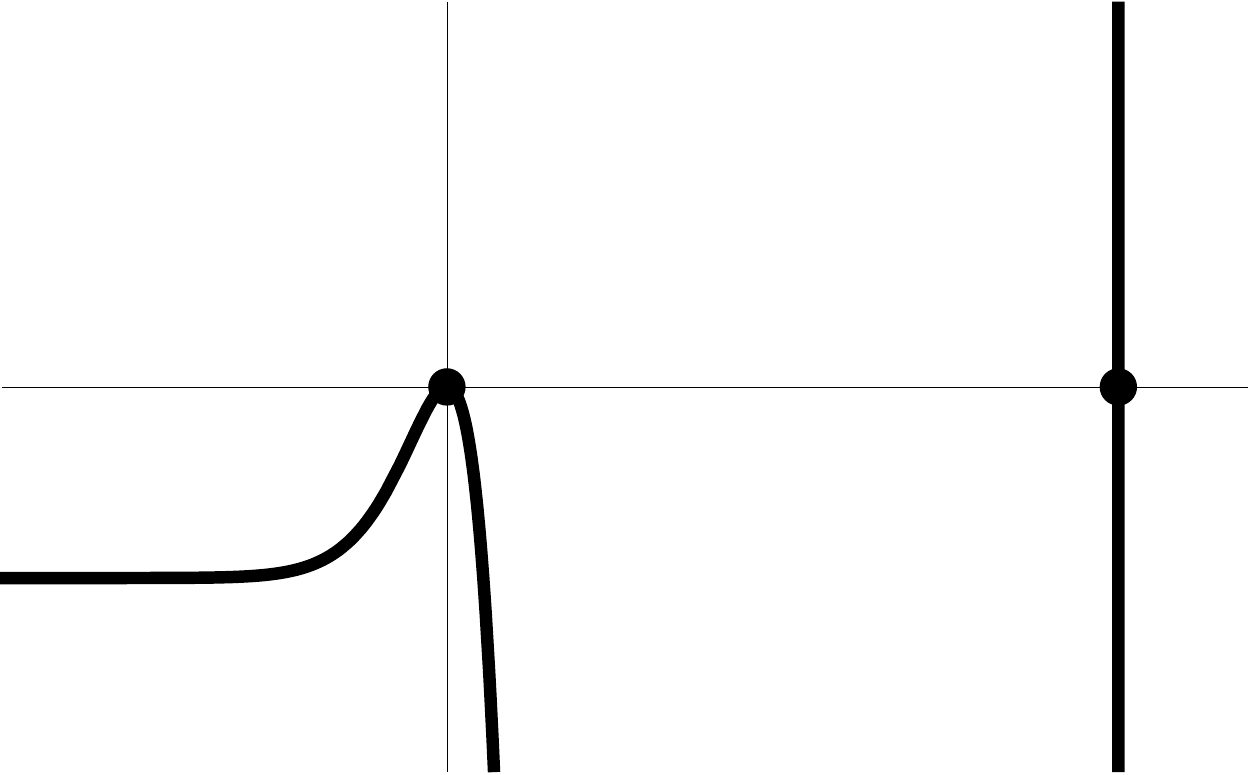}};
            \node[right=of img, node distance=0cm, yshift=0cm,xshift=-1cm] {$x$};
            \node[above=of img, node distance=0cm, anchor=center, yshift=-0.8cm,xshift=-0.8cm] {$g(x)$};
            \node[below=of img, node distance=1cm, yshift=3.4cm,xshift=-1.15cm] {$0$};
            \node[below=of img, node distance=1cm, yshift=3.55cm,xshift=2.2cm] {$d_0$};
            \end{tikzpicture}
            \caption{An example of $g(x)$ from the proof Lemma~\ref{lem:dr} for $r = 0$.  In particular, an analysis of $g(x)$ and $g'(x)$ show that there are exactly two roots which are at $x = 0$ and $x = d_0$. }
            \label{fig:gxRoots}
        \end{figure}
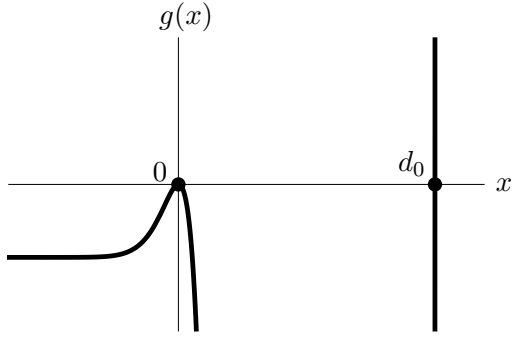

        It remains only to demonstrate the uniqueness of $d_r$. Consider the derivative of $g$:
            \begin{equation*}
            g'(x) = (r-x)f''(x).
            \end{equation*}
        Since $f$ is strictly convex over $[0,c]$ and strictly concave over $[c,1]$, it follows that
            \begin{equation*}
            \begin{cases}
            g'(x) > 0 & \textrm{if } x \in [0,r), \\
            g'(x) < 0 & \textrm{if } x \in (r,c), \\
            g'(c) = 0, \\
            g'(x) > 0 & \textrm{if } x \in (c,1], \\
            \end{cases}
            \end{equation*}
        which implies that our roots are unique. See Figure \ref{fig:gxRoots}.
    \end{proof}
\end{framed}

\begin{prop}
\label{prop:Compute_d_0}
For any $r \in [0,1]$, we can compute $d_r$ up to $\epsilon$ additive error in $O(\log(\tfrac{1}{\epsilon}))$ many oracle calls of $f$ and $f'$.
\end{prop}
\begin{proof}
 Define $g(x)$ as in the previous proof.
    By the previous lemma, $g(r) = 0$ and $g(d_r) = 0$. Furthermore, $d_r$ lies between $2c - r$ and $c$.   Since $d_r$ is unique, $g$ doesn't have any other zeros.  Thus, we can run the bisection method with initial bounds $2c - r$ and $c$ to obtain the desired complexity.  Note that each evaluation of $g$ requires evaluations of both $f'$ and $f$.
\end{proof}

\begin{theorem}[KKT Solutions of \kspace]
\label{thm:KKT-Solutions}
Suppose that $f$ is twice-differentiable, antisymmetric about $c$, and strictly convex on $[0,c)$. The KKT solutions to the Continuous Relaxation of \kspace are given in the following table.
    $$
    \begin{array}{cc|cccc}
    \textbf{Case} & \textbf{Condition}       & k_0                  & k_1                  & k_y                & y           \\\hline
    \textbf{\caseZero}    & k_y = 0         & n-M     & M     & 0     & \sim       \\
    \textbf{\caseOne}    & k_0,k_1,k_y > 0 & \multicolumn{4}{c}{\text{No KKT Solution}}                                                 \\
    \textbf{\caseTwoA}   & k_y = n         & 0                    & 0                    & n                  & \frac{M}{n} \\
    \textbf{\caseTwoB}   & k_1 = 0         & n-\frac{M}{d_0} & 0                    & \frac{M}{d_0} & d_0         \\
    \textbf{\caseTwoC}   & k_0 = 0          & 0                    & n-\frac{n-M}{1-d_1} & \frac{n-M}{1-d_1} & d_1        
    \end{array}
    $$
\end{theorem}

Note that the ``$\sim$'' symbol under case \caseZero\ represents an undefined value for $y$. This is because $k_y = 0$ and, by equation \ref{eq:yInTermsOfky}, $y$ is undefined for all such solutions.

\begin{framed}
    \begin{proof}
        The Lagrangian function is
        \begin{equation}\begin{split} 
        \mathcal L(\mathbf{k}, \lambda, \mu)
        &= f(0) k_0 + f(1) k_1 + f(y) k_y  + \lambda_1 (k_1 + y k_y - M) + \lambda_2 (k_0 + k_1 + k_y - n) \\
            & + \mu_0 k_0  + \mu_1 k_1  + \mu_y k_y 
            + \xi^-(y) + \xi^+(1 - y)
        \end{split}\end{equation}
        
        The corresponding KKT conditions are\\
         \begin{subequations}
        \begin{minipage}[t]{.45\textwidth}
        \vspace{-0.5cm}
        \begin{align}
            f(0) + \lambda_2 + \mu_0 &= 0 \label{partialka}\\
            f(1) + \lambda_1 + \lambda_2 + \mu_1 &= 0\label{partialkb}\\
            f(y) + \lambda_1 y + \lambda_2 + \mu_y &= 0 \label{partialky}\\
            k_y f'(y) + \lambda_1 k_y + \xi^- - \xi^+&= 0 \label{partialy}\\
           \forall\ i \in \{0,y,1\} \qquad \mu_i \cdot k_i &= 0 &  \label{slackmu}\\
            y\xi^-  &= 0 \label{slackxi-}\\
            (1-y)\xi^+  &= 0 \label{slackxi+}
        \end{align}
        \end{minipage}
          \begin{minipage}[t]{.45\textwidth}
          \begin{align}
            k_1+yk_y &= M \label{feasibleM}\\
            k_0+k_1+k_y &= n \label{feasiblen}\\
            k_0, k_1, k_y & \geq 0\\
            0 \leq y & \leq 1 \\
            \forall\ i \in \{0,y,1\} \qquad \mu_i &\geq 0  &\\
            \xi^-,\xi^+ &\geq 0
            \end{align}
        \end{minipage}
                \end{subequations}

        \begin{itemize}
            \item[\textbf{\caseZeroA}.] \big[$k_y = 0$\big]
            Then \eqref{feasibleM} reduces to $k_1 = M$ and \eqref{feasiblen} uniquely gives $k_0 = n-M$.
            
            \item[\textbf{\caseZeroB.}] \big[$y = 0$ or $y = 1$\big]
            Then $k_y$ can be combined with $k_0$ or $k_1$, respectively, reducing the problem to these two variables and giving the same outcome as Case~\caseZeroA.
        \end{itemize}
    
        \noindent We may now assume that $0 < y < 1$ and $k_y > 0$ so that \eqref{slackxi-} and \eqref{slackxi+} require $\xi^-=0$ and  $\xi^+ = 0$. Therefore \eqref{partialy} reduces to
            $$
            k_y f'(y) + \lambda_1 k_y = 0 \quad\Rightarrow\quad \lambda_1 = -f'(y)
            $$
        Then by \eqref{partialka}, \eqref{partialkb}, and \eqref{partialky}, solving for $\lambda_2$, we have
            \begin{equation*}
            \begin{array}{lcccr}
            f(0) + \mu_0
            & =&   f(1) - f'(y) + \mu_1
            & =&   f(y) - f'(y) y + \mu_y
            \end{array}
            \end{equation*}
        which can be rewritten as
        \begin{subequations}
        \begin{align}
            f(1) - f(0) - f'(y) + (\mu_1 - \mu_0) &=0, \\
            f(1) - f(y) - f'(y)(1-y) + (\mu_1 - \mu_y) &= 0,\label{mu23} \\
            f(y) - f(0) - yf'(y) + (\mu_y - \mu_0) &= 0. \label{mu31}
        \end{align}
        \end{subequations}
        Note that one of these equations is redundant.
        
        \begin{itemize}
            \item[\caseOne] \big[$k_0, k_1, k_y > 0$\big]
            Then by \eqref{slackmu}, $\mu_i = 0$ for $i \in \{0,y,1\}$ and equations \eqref{mu23} and \eqref{mu31} give
                \begin{equation*}
                f(1) - f(y) - f'(y)(1-y)  = 0 \qquad\text{and}\qquad f(y) - f(0) - yf'(y)  = 0.
                \end{equation*}
            According to Lemma \ref{lem:dr}, the unique solutions to these equations, respectively,  are $y = d_1$ and $y = d_0$ ($y=0$ and $y=1$ are also roots but have already been excluded by case \caseZeroB). Also by Lemma \ref{lem:dr}, it must be that $d_1 < c < d_0$ because $0 < c < 1$. Thus $d_0 \neq d_1$, so there is no solution satisfying both equations and there is no KKT appropriate solution in this case.
            
            \item[\textbf{\caseTwoA.}] \big[$k_y = n$\big]
            Then $k_0 = k_1 = 0$ by \eqref{feasiblen}, so \eqref{feasibleM} simply reduces to $y = \tfrac{M}{n}$. 
            
            \item[\textbf{\caseTwoB.}] \big[$k_1 = 0$, $k_0 > 0$, $k_y > 0$\big]
            Then $\mu_0 = \mu_y = 0$.  Hence \eqref{mu31} reduces to $f(y) - f(0) - yf'(y) = 0$ which has unique solution $y = d_0$ by Lemma~\ref{lem:dr}. Thus \eqref{feasibleM} reduces to $k_y = \frac{M}{d_0}$ and \eqref{feasiblen} uniquely gives $k_0 = n-k_y = n-\frac{M}{d_0}$.
            
            \item[\textbf{\caseTwoC.}] \big[$k_0 = 0$, $k_1 > 0$, $k_y > 0$\big]
            Then $\mu_1 = \mu_y = 0$.  Hence \eqref{mu23} reduces to $f(1)-f(y)-f'(y)(1-y) = 0$ which has unique the solution $y = d_1$ by Lemma~\ref{lem:dr}. Thus \eqref{feasibleM} and \eqref{feasiblen} give $k_1+d_1k_y = M$ and $k_1+k_y = n$ which are solved by $k_1=n-\frac{n-M}{1-d_1}$ and $k_y = \frac{n-M}{1-d_1}$.
        \end{itemize}
        This concludes the case analysis.
    \end{proof}
\end{framed}

Using \eqref{eq:yInTermsOfky}, we can uniquely project an equality-feasible solution $\mathbf{k} = (k_0,k_1,k_y,y)$ onto any two of its parameters. Let $\brObj(k_0,k_1)$ be the projection of $F(\mathbf{k})$ onto the $(k_0,k_1)$ space. Under this projection, both $y$ and $k_y$ are functions of $k_0$ and $k_1$. Figure \ref{fig:gradient} contains plots the vector field of the $\nabla \brObj(k_0,k_1)$ (the value of $y$ is represented by background color gradient). Theorems \ref{thm:y=d_0} and \ref{thm:k_1isextreme} are informed by Figure~\ref{fig:gradient}. 

\renewcommand{\scaling}{0.84}
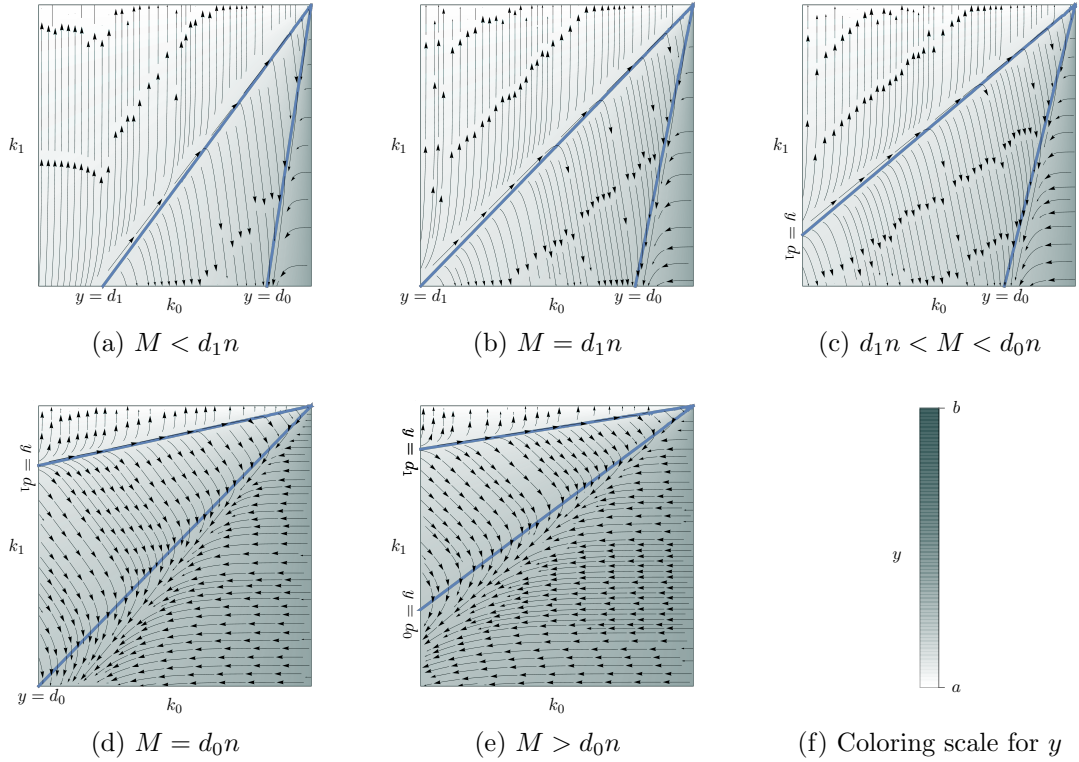
\begin{figure}
    \begin{subfigure}{0.3\textwidth}\label{fig:gradient:M<dn}
    \resizebox{\scaling\textwidth}{\scaling\textwidth}{%
        \begin{tikzpicture}
        \node (img)  {\includegraphics[scale=0.225]{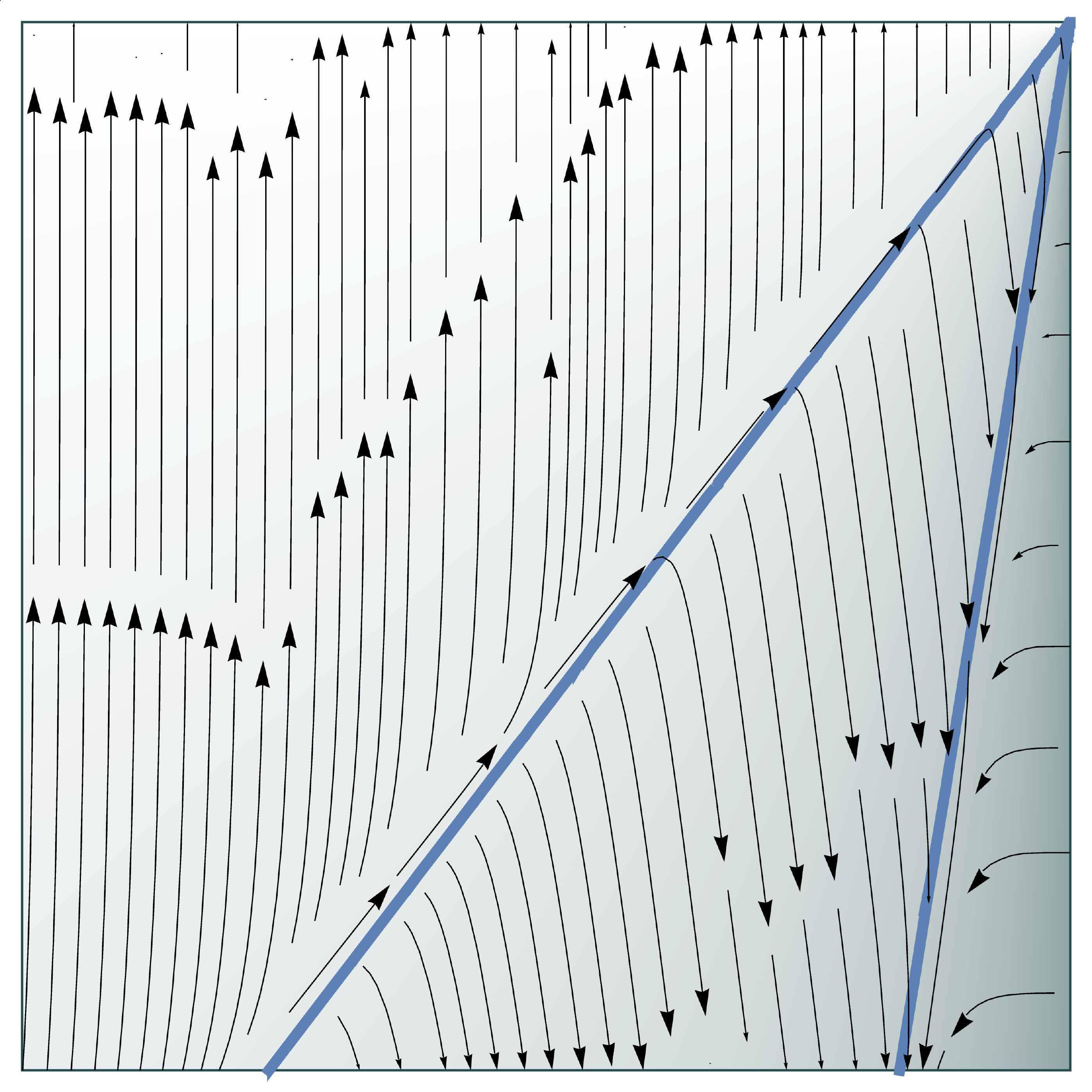}};
        \node[below=of img, node distance=0cm, yshift=1cm] {$k_0$};
        \node[below=of img, node distance=1cm, yshift=1.2cm,xshift=2.1cm] {$y=d_0$};
        \node[below=of img, node distance=1cm, yshift=1.2cm,xshift=-1.7cm] {$y=d_1$};
        \node[left=of img, node distance=0cm, anchor=center,xshift=0.7cm] {$k_1$};
        \end{tikzpicture}
}
\centering\caption{$M < d_1n$}
    \end{subfigure}
    \begin{subfigure}{0.3\textwidth}\label{fig:gradient:M=dn}
    \resizebox{\scaling\textwidth}{\scaling\textwidth}{%
        \begin{tikzpicture}
        \node (img)  {\includegraphics[scale=0.225]{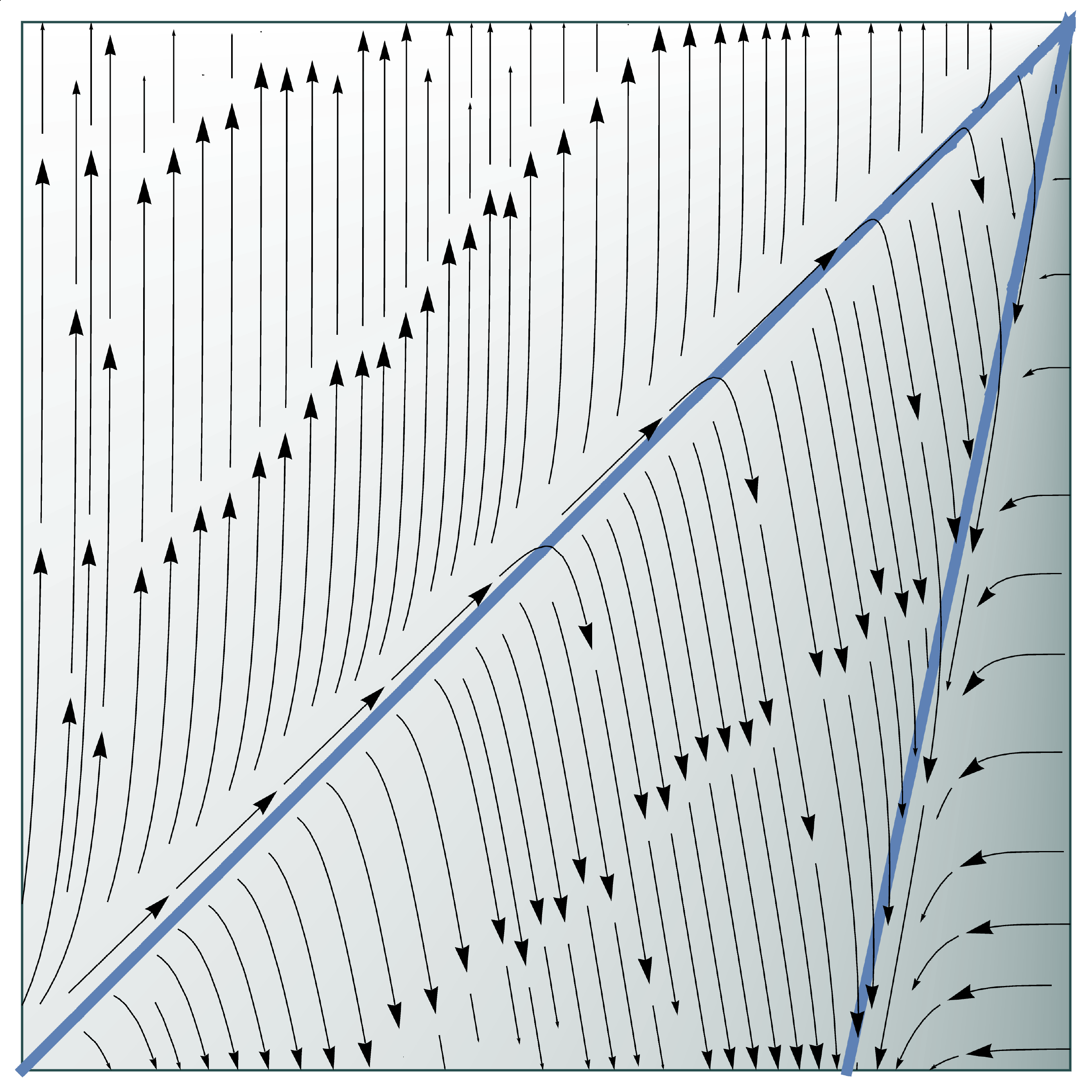}};
        \node[below=of img, node distance=0cm, yshift=1cm] {$k_0$};
        \node[below=of img, node distance=1cm, yshift=1.2cm,xshift=1.84cm] {$y=d_0$};
        \node[below=of img, node distance=1cm, yshift=1.2cm,xshift=-3cm] {$y=d_1$};
        \node[left=of img, node distance=0cm, anchor=center,xshift=0.7cm] {$k_1$};
        \end{tikzpicture}
        }
        \centering\caption{$M = d_1n$}
    \end{subfigure}
    \begin{subfigure}{0.3\textwidth}\label{fig:gradient:M<dan}
    \resizebox{\scaling\textwidth}{\scaling\textwidth}{%
        \begin{tikzpicture}
        \node (img)  {\includegraphics[scale=0.225]{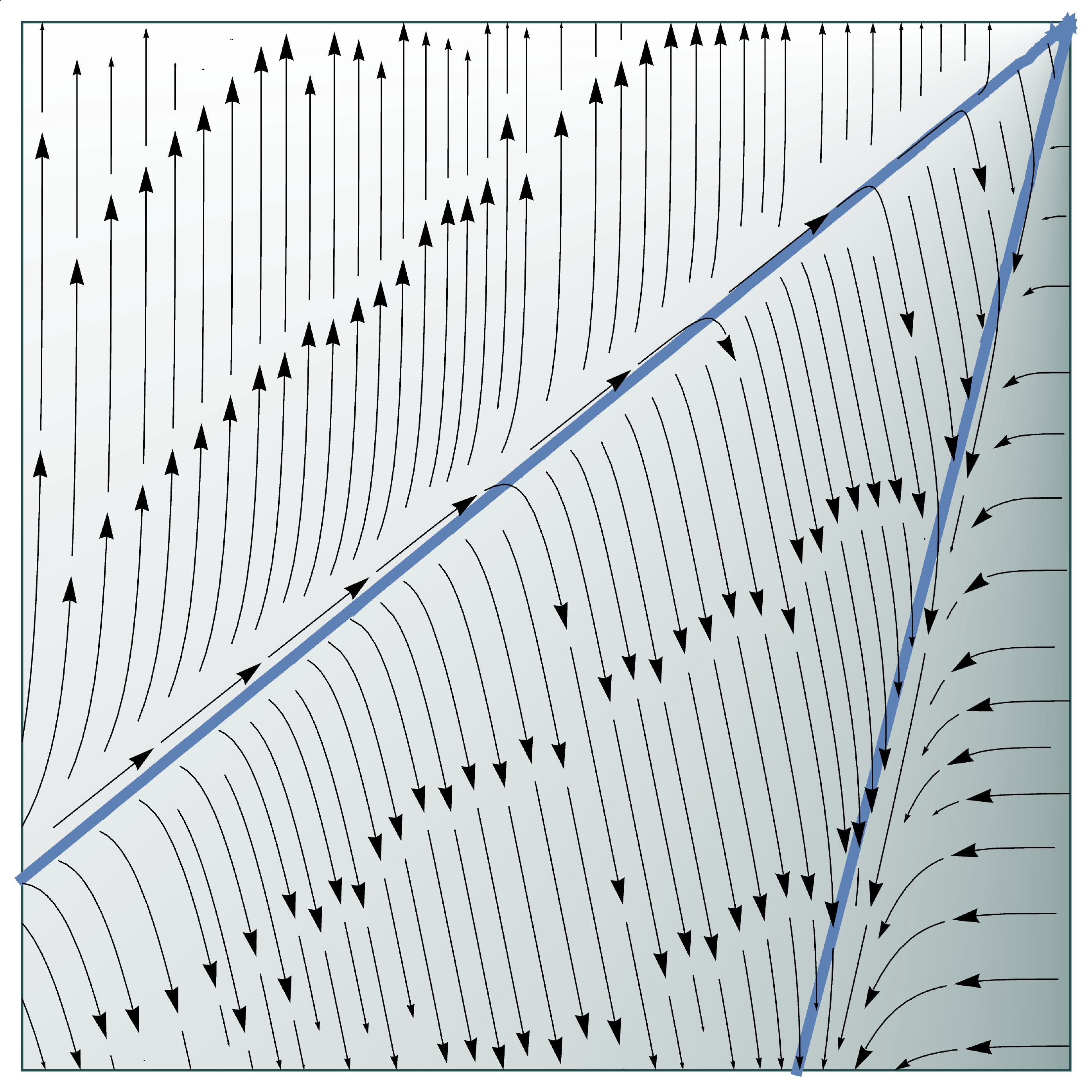}};
        \node[below=of img, node distance=0cm, yshift=1cm] {$k_0$};
        \node[below=of img, node distance=1cm, yshift=1.2cm,xshift=1.5cm] {$y=d_0$};
        \node[left=of img, node distance=0cm, anchor=center,xshift=0.7cm] {$k_1$};
        \node[left=of img, node distance=0cm, rotate=-90, anchor=center, yshift=0.95cm, xshift=2cm] {$y=d_1$};
        \end{tikzpicture}
        }
        \centering\caption{$d_1n < M < d_0n$}
    \end{subfigure}
    \, \vspace{0.5cm} \, 
    
    \begin{subfigure}{0.3\textwidth}\label{fig:gradient:M=dan}
    \resizebox{\scaling\textwidth}{\scaling\textwidth}{%
        \begin{tikzpicture}
        \node (img)  {\includegraphics[scale=0.225]{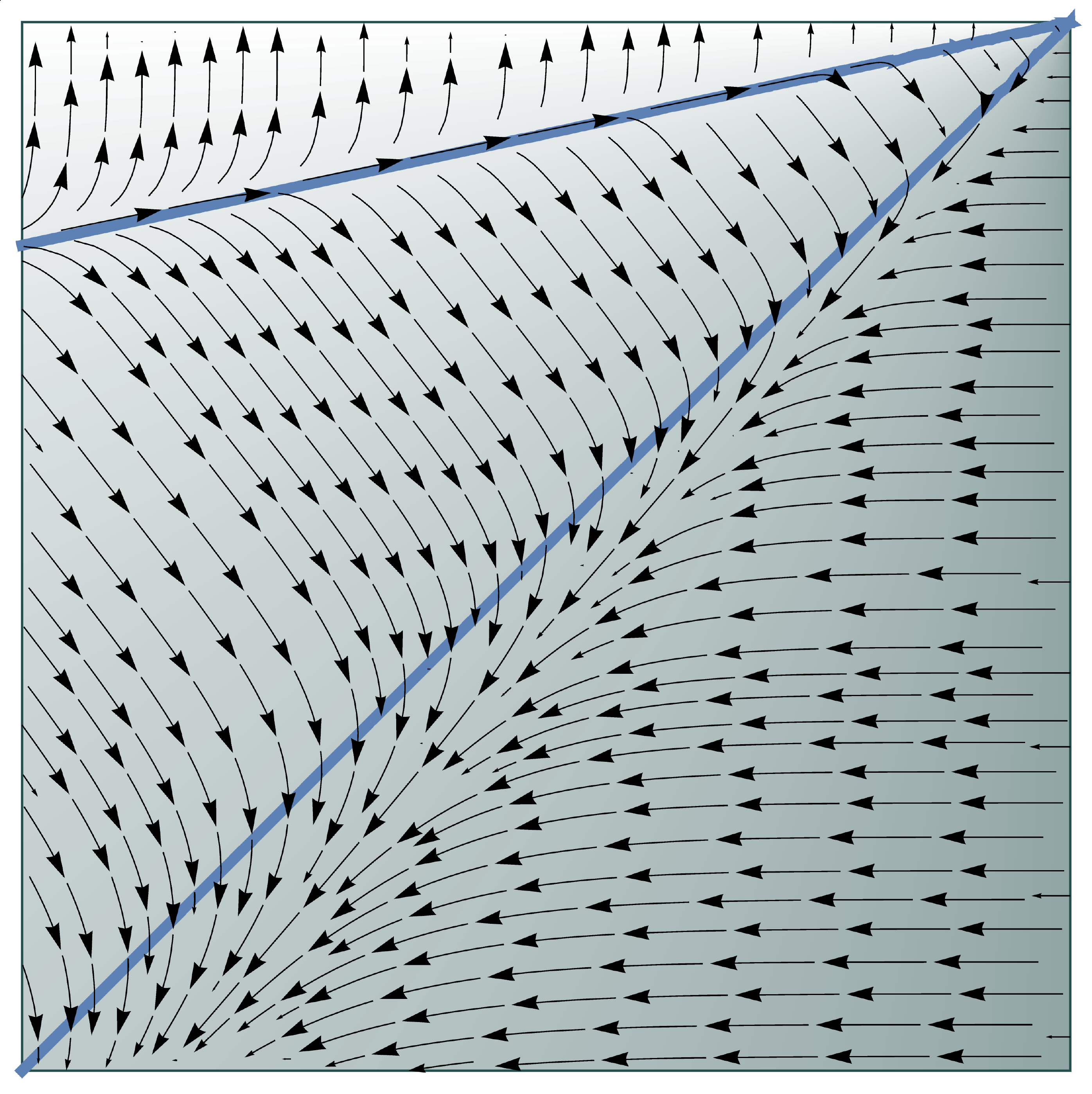}};
        \node[below=of img, node distance=0cm, yshift=1cm] {$k_0$};
        \node[below=of img, node distance=1cm, yshift=1.2cm,xshift=-3cm] {$y=d_0$};
        \node[left=of img, node distance=0cm, anchor=center,xshift=0.7cm] {$k_1$};
        \node[left=of img, node distance=0cm, rotate=-90, anchor=center, yshift=0.95cm, xshift=-1.7cm] {$y=d_1$};
        \end{tikzpicture}
        }
        \centering\caption{$M = d_0n$}
    \end{subfigure}
    \begin{subfigure}{0.3\textwidth}\label{fig:gradient:M>dan}
    \resizebox{\scaling\textwidth}{\scaling\textwidth}{%
        \begin{tikzpicture}
        \node (img)  {\includegraphics[scale=0.225]{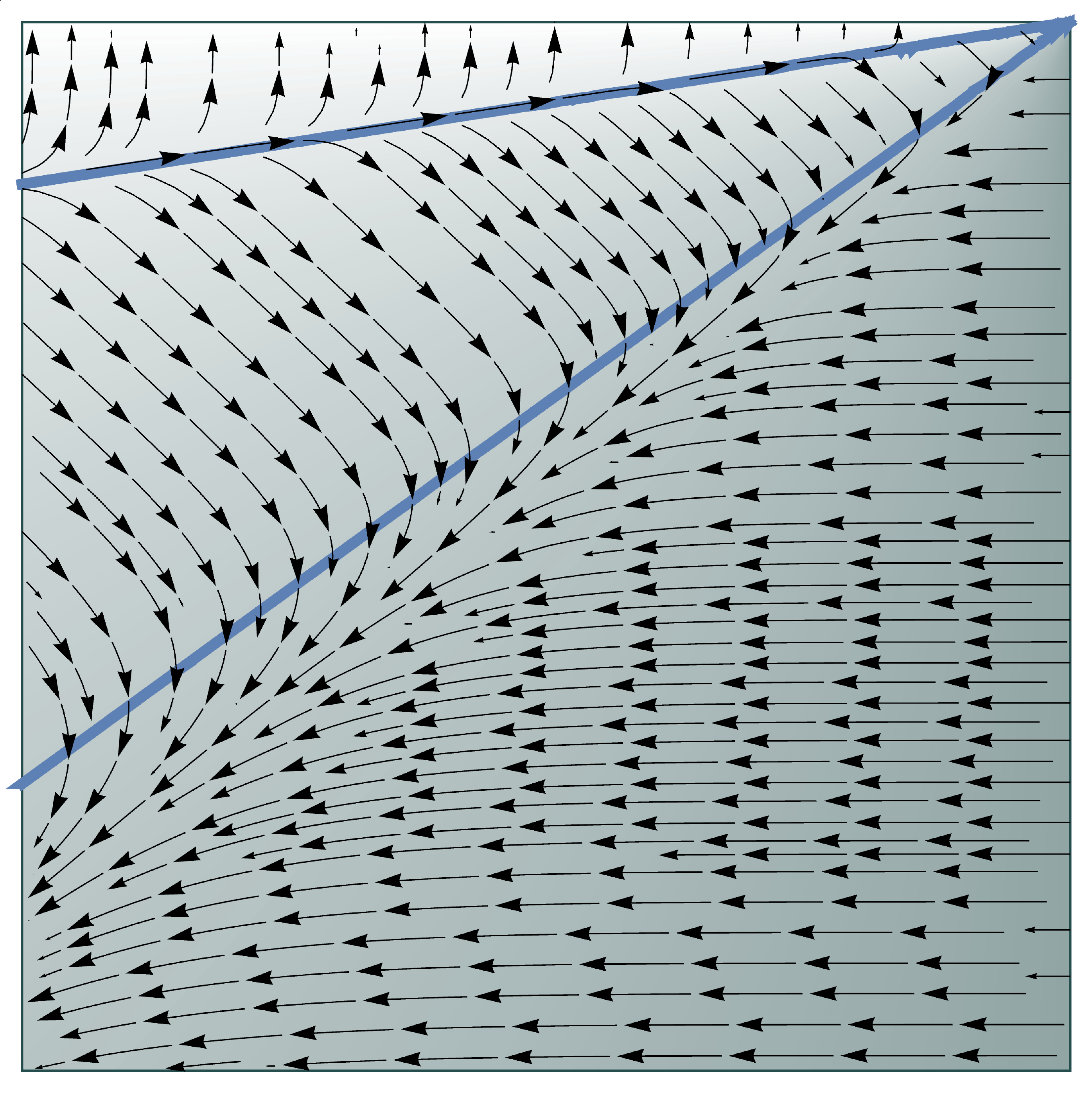}};
        \node[below=of img, node distance=0cm, yshift=1cm] {$k_0$};
        \node[left=of img, node distance=0cm, anchor=center,xshift=0.7cm] {$k_1$};
        \node[left=of img, node distance=0cm, rotate=-90, anchor=center, yshift=0.95cm, xshift=1.5cm] {$y=d_0$};
        \node[left=of img, node distance=0cm, rotate=-90, anchor=center, yshift=0.95cm, xshift=-2.05cm] {$y=d_1$};
        \node[left=of img, node distance=0cm, rotate=-90, anchor=center, yshift=0.95cm, xshift=-2.05cm] {$y=d_1$};
        \end{tikzpicture}
        }
        \centering\caption{$M > d_0n$}
    \end{subfigure}
    \begin{subfigure}{0.3\textwidth}\centering
    \resizebox{0.2\textwidth}{\scaling\textwidth}{%
    \begin{tikzpicture}
        \node (img)  {\includegraphics[scale=0.87]{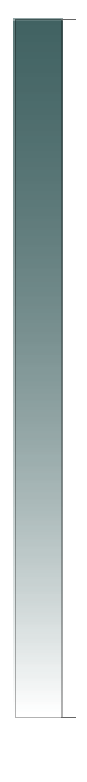}};
        \node[left=of img, node distance=0cm, anchor=center, xshift=0.7cm] {$y$};
        \node[right=of img, node distance=0cm, anchor=center, xshift=-0.9cm, yshift=-2.9cm] {$a$};
        \node[right=of img, node distance=0cm, anchor=center, xshift=-0.9cm, yshift=3.3cm] {$b$};
        \end{tikzpicture}
        }
        \centering\caption{Coloring scale for $y$}
    \end{subfigure}
    \caption{Stream plots of the gradient field of $\brObj(k_0,k_1)$.  Various cases are exhibited that expose different KKT solutions as possible optimal solutions.}\label{fig:gradient}
\end{figure}

\begin{theorem}\label{thm:y=d_0}
    Suppose that $f$ is twice-differentiable, antisymmetric about $c$, and strictly convex on $[0,c)$ and let $\mathbf{k} = (k_0,k_1,k_y,y)$ be an equality-feasible solution of \kspace. We have
        \begin{equation*}
        \pderiv{}{k_0}y \geq 0 \qquad\text{and}\qquad
            \begin{cases}
            \pderiv{}{k_0}\brObj(k_0,k_1) = 0, & \text{if } y = d_0,         \\
            \pderiv{}{k_0}\brObj(k_0,k_1) > 0, & \text{if } y \in (0,d_0),   \\
            \pderiv{}{k_0}\brObj(k_0,k_1) < 0, & \text{if } y \in (d_0,1).
            \end{cases}
        \end{equation*}
\end{theorem}
\begin{proof}
    Recall from equation \eqref{eq:yInTermsOfky} that since $\mathbf{k}$ is equality-feasible, we can treat $y$ and $k_y$ as functions of $k_0$ and $k_1$ from \eqref{constr:kspaceM} and \eqref{constr:kspacen}
        \begin{equation*}
        k_y = n-k_0-k_1 \qquad\textrm{and}\qquad y = \tfrac{M-k_1}{k_y}.
        \end{equation*}
    Clearly $\pderiv{}{k_0}k_y = \pderiv{}{k_0}(n-k_0-k_1) = -1$. It follows that
        \begin{equation*}
            \pderiv{}{k_0}y \quad=\quad \pderiv{}{k_0}\left(\tfrac{M-k_1}{k_y}\right)   
                            \quad=\quad \tfrac{M-k_1}{k_y^2}      
                            \quad=\quad \tfrac{y}{k_y}.
        \end{equation*}
    which is non-negative for any $y \in (0,1)$ and $k_y > 0$. 
    
    Now, notice that the partial derivative
        \begin{align*}
        \pderiv{}{k_0}\brObj(k_0,k_1) &= \pderiv{}{k_0}\left(f(0)k_0+f(1)k_1+f(y)k_y\right)  
            = f(0)+f'(y)\left(\pderiv{}{k_0}y\right)k_y+f(y)\pderiv{}{k_0}k_y                  \\
            &= f(0)+yf'(y)-f(y)
        \end{align*}
    equals zero if $y = 0$ or $y = d_0$. These are the only roots since $d_0$ is unique. First, suppose $y=c$ and notice that $\pderiv{}{k_0}\brObj(k_0,k_1) = f(0)+cf'(c)-f(c)$ is positive by the strict convexity of $f(x)$ over $x\in[0,c)$. Now suppose that $y=\overline{0}=2c$. By the antisymmetry of $f$, we have
        \begin{align*}
        \pderiv{}{k_0}\brObj(k_0,k_1) 
            &= f(0)+2cf'(2c)-f(2c)
            = f(0)+2cf'(0)-\left(2f(c)-f(0)\right) \\
            &= 2\left(f(0)+cf'(0)-f(c)\right)
        \end{align*}
    which is negative by the strict convexity of $f(x)$ over $x\in[0,c)$.
\end{proof}
This theorem implies that: given a solution with $y < d_0$, increasing $k_0$ increases both $y$ and $\brObj(\mathbf{k})$; or given a solution with $y > d_0$, increasing $k_0$ increases $y$ but decreases $\brObj(\mathbf{k})$. In either case, we can perturb $k_0$ to any feasible extent because $y$ will not leave $(0,1)$. We can see this behavior from Figure \ref{fig:gradient} by noticing that the $k_0$ component of each vector points towards the line $y=d_0$. Similarly, the $k_1$ component of each vector points away from the line $y=d_1$ which informs the next theorem.

\begin{theorem}\label{thm:k_1isextreme}
    Suppose that $f$ is twice-differentiable, antisymmetric about $c$, and strictly convex on $[0,c)$ and let $\mathbf{k} = (k_0,k_1,k_y,y)$ be an equality-feasible solution of \eqref{obj:kspace}. We have
        \begin{equation*}
        \pderiv{}{k_1}y \leq 0 \qquad\text{and}\qquad
            \begin{cases}{}
            \pderiv{}{k_1}\brObj(\mathbf{k}) = 0, & \text{if } y = d_1         \\
            \pderiv{}{k_1}\brObj(\mathbf{k}) > 0, & \text{if } y \in (0,d_1)   \\
            \pderiv{}{k_1}\brObj(\mathbf{k}) < 0, & \text{if } y \in (d_1,1).
            \end{cases}
        \end{equation*}
        
\end{theorem}
The proof of this parallels that of Theorem~\ref{thm:y=d_0}. With the behavior of the gradient now well understood, we can find the optimal solution.

\begin{corollary}[Ordering and feasibility of KKT solutions]\label{thm:continuousoptimal}
Suppose that $f$ is twice-differentiable, antisymmetric about $c$, and strictly convex on $[0,c)$. The optimal solution of the continuous relaxation of \kspace is given by
    \begin{equation*}
    (k_0^*,k_1^*,k_y^*,y^*) = 
        \begin{cases}
        \left(n-\frac{M}{d_0},0,\frac{M}{d_0},d_0\right) &\text{if } M \leq d_0n, \\
        \left(0,0,n,\frac{M}{n}\right) &\text{if } M > d_0n.
        \end{cases}
    \end{equation*}
\end{corollary}
\begin{proof}
    It suffices to investigate the KKT solutions given in the table in Theorem~\ref{thm:KKT-Solutions}. Thus we consider the following solutions:
    \begin{align*}
        \mathbf{k}_\textbf{\caseZero} &= \begin{bmatrix}
            n-M     & M     & 0  & \sim
            \end{bmatrix} &
        \mathbf{k}_\caseTwoA &= \begin{bmatrix}
            0   & 0  & n   & \frac{M}{n}
            \end{bmatrix} \\
        \mathbf{k}_\caseTwoB &= \begin{bmatrix}
            n-\frac{M}{d_0} & 0    & \frac{M}{d_0} & d_0
            \end{bmatrix} &
        \mathbf{k}_\caseTwoC &= \begin{bmatrix}
            0   & n-\frac{n-M}{1-d_1} & \frac{n-M}{1-d_1} & d_1  
            \end{bmatrix}
    \end{align*}
    Note that the ``$\sim$'' symbol in $\mathbf{k}_\textbf{\caseZero}$ represents an undefined value for $y$. This is because $k_y = 0$ and, by equation \ref{eq:yInTermsOfky}, $y$ is undefined for all such solutions.
    
    First, let's observe the feasibility of all these solutions: $\mathbf{k}_\textbf{\caseZero}$ and $\mathbf{k}_\caseTwoA$ are feasible whenever $M \in [0,n]$ while $\mathbf{k}_\caseTwoB$ is only feasible for $M\in[0,d_0n]$ and $\mathbf{k}_\caseTwoC$ is only feasible for $M\in[d_1n,n]$. 

    Suppose that $M \in [d_1 n, n]$.  
    Notice that 
        $$
\brObj\left(\mathbf{k}_{\caseTwoC}\right) = \brObj\left(0,n-\tfrac{n-M}{1-d_1},\tfrac{n-M}{1-d_1},d_1\right) < \brObj\left(0,M,n-M,y^*\right)
        $$
    by Theorem \ref{thm:k_1isextreme} since $k_0$ is kept constant while $k_1$ is increased. By \eqref{eq:yInTermsOfky}, we must have $y^* = 0$ to maintain equality-feasibility. Thus, in this case, we can exchange $k_y$ for $k_0$ rewriting the latter solution  equivalently as $\mathbf{k}_\caseZero$.
     Thus $\brObj\left(\mathbf{k}_{\caseZero}\right) > \brObj\left(\mathbf{k}_{\caseTwoC}\right)$ over $M\in[d_1n,n]$
    
    It remains only to show that: \eqref{arg:2bDominates} $\mathbf{k}_\caseTwoB$ dominates over $M\in[0,d_0n]$ and \eqref{arg:2aDominates} $\mathbf{k}_\caseTwoA$ dominates over $M \in (d_0n,n]$.

    \begin{enumerate}[I.]
        \item\label{arg:2bDominates}\big[$\mathbf{k}_\caseTwoB$ dominates over $M\in[0,d_0n]$\big]
        \begin{enumerate}[i.]
            \item\label{arg:2bBetterThan2a} 
            \big[$\brObj\left(\mathbf{k}_{\caseTwoB}\right) \geq \brObj\left(\mathbf{k}_{\caseTwoA}\right)$\big]\quad Since $\frac{M}{n} \geq d_0$,
                $$
                \brObj\left(\mathbf{k}_{\caseTwoB}\right) = 
                \brObj\left(n-\tfrac{M}{d_0},0,\tfrac{M}{d_0},d_0\right) \geq \brObj\left(0,0,n,\tfrac{M}{n}\right) = 
                \brObj\left(\mathbf{k}_{\caseTwoA}\right)
                $$
            by Theorem~\ref{thm:y=d_0} since $k_1$ is kept constant while $k_0$ is decreased. (\textit{Note:} the inequality here is strict unless the solutions coincide, which occurs  when $M = d_0n$.)

            \item\label{arg:2bBetterThan0}
            \big[$\brObj\left(\mathbf{k}_{\caseTwoB}\right) > \brObj\left(\mathbf{k}_\caseZero\right)$\big]\quad By Theorem \ref{thm:y=d_0},
                $$
                \brObj\left(\mathbf{k}_{\caseTwoB}\right) = \brObj\left(n-\tfrac{M}{d_0},0,\tfrac{M}{d_0},d_0\right)>
                \brObj\left(n-M,0,M,y^*\right)
                $$
            since $k_1$ is kept constant while $k_0$ is increased and $y^* > d_0$. By \eqref{eq:yInTermsOfky}, we must have $y^* = 1$ to maintain equality-feasibility. Thus, in this case, we can exchange $k_y$ for $k_1$ rewriting the latter solution  equivalently as $\mathbf{k}_\caseZero$.

            \item\label{arg:2bBetterThan2c} 
            \big[$\brObj\left(\mathbf{k}_{\caseTwoB}\right) > \brObj\left(\mathbf{k}_\caseTwoC\right)$ for $M\in[d_1n,d_0n]$\big]\quad Follows from \eqref{arg:2bBetterThan0} since $\brObj\left(\mathbf{k}_{\caseZero}\right) > \brObj\left(\mathbf{k}_{\caseTwoC}\right)$.
        \end{enumerate}

        \item\label{arg:2aDominates}
        \big[$\mathbf{k}_\caseTwoA$ dominates over $M \in (d_0n,n]$\big]
        \begin{enumerate}[i.]
            \item\label{arg:2aBetterThan0}
            \big[$\brObj\left(\mathbf{k}_{\caseTwoA}\right) > \brObj\left(\mathbf{k}_{\caseZero}\right)$\big]\quad Notice that
                $$
                \brObj\left(\mathbf{k}_{\caseTwoA}\right) = 
                \brObj\left(0,0,n,\tfrac{M}{n}\right) > \brObj\left(n-M,0,M,y^*\right)
                $$
            by Theorem \ref{thm:y=d_0} since $\frac{M}{n} > d_0$ and $k_1$ is kept constant while $k_0$ is increased.  By \eqref{eq:yInTermsOfky}, we must have $y^* = 1$ to maintain equality-feasibility. Thus, again, we rewrite the latter solution  equivalently as $\mathbf{k}_\caseZero$.

            \item\label{arg:2aBetterThan2c}
            \big[$\brObj\left(\mathbf{k}_{\caseTwoA}\right) > \brObj\left(\mathbf{k}_{\caseTwoC}\right)$\big]\quad Follows from  \eqref{arg:2aBetterThan0} since $\brObj\left(\mathbf{k}_{\caseZero}\right) > \brObj\left(\mathbf{k}_{\caseTwoC}\right)$.
        \end{enumerate}
    \end{enumerate}
\end{proof}

\subsection{Integer Solutions}%

\begin{corollary}[to Lemma \ref{lem:feasibility}]
\label{lem:kakbintfeasibility}
An equality feasible solution $(k_0,k_1,k_y,y)$ is fully feasible for \kspace  if and only if $k_0$ and $k_1$ are integers such that
\begin{equation*}
\begin{array}{ccccccccccc}
\max\left(    \left\lceil\frac{yn-M}{y}\right\rceil, 0\right) &\leq& k_0 &\leq& \left\lfloor n-M\right\rfloor     &\qquad\textrm{and}\qquad&  \max\left(  \left\lceil\frac{M-yn}{1-y}\right\rceil , 0 \right)&\leq& k_1 &\leq& \left\lfloor M \right\rfloor
\end{array}
\end{equation*}
\end{corollary}
The result follows directly from  Lemma \ref{lem:feasibility} when we restrict  $k_0, k_1$ and $k_y$ to be integer.

\begin{theorem}\label{thm:opt_int_sol}
    Suppose that $f$ is twice-differentiable, antisymmetric about $c$, and strictly convex on $[0,c)$. The optimal integer solution of \kspace is given by
        \begin{equation*}
        (k_0^*, k_1^*, k_y^*, y^*) = \begin{cases}
            \text{best of } \mathbf{k}_\caseZeroMinus, \mathbf{k}_\caseTwoBMinus, \mathbf{k}_\caseTwoBPlus & \text{if } M < d_0n,\\
            \mathbf{k}_\caseTwoA & \text{if } M \geq d_0 n.
            \end{cases}
        \end{equation*}
    Where $\mathbf{k}_\caseZeroMinus$, $\mathbf{k}_\caseTwoBMinus$, $\mathbf{k}_\caseTwoBPlus$, and $\mathbf{k}_\caseTwoA$ are defined the following table:
    \[\arraycolsep=5pt\def\arraystretch{1.5}
\begin{array}{c|cccc}
\textbf{Point} \        
& k_0                  & k_1                  & k_y                & y           \\\hline
\mathbf{k}_\caseZeroMinus  & \left\lfloor n-M\right\rfloor    &\left\lfloor M\right\rfloor    & n-k_0-k_1       & M-\left\lfloor M\right\rfloor    \\
\mathbf{k}_\caseTwoA            & 0                                                  & 0     & n                                                 & \frac{M}{n}   \\
\mathbf{k}_\caseTwoBMinus & \left\lfloor n-\frac{M}{d_0}\right\rfloor   & 0     & \left\lceil\frac{M}{d_0}\right\rceil    & \frac{M}{k_y} < d_0         \\
\mathbf{k}_\caseTwoBPlus & \left\lceil n-\frac{M}{d_0}\right\rceil     & 0     & \left\lfloor\frac{M}{d_0}\right\rfloor  & \frac{M}{k_y} > d_0       
\end{array}
\]
Thus, an optimal solution can be computed with a constant number of operations, provided a preprocessing algorithm to determine $d_0$ with sufficient accuracy.
\end{theorem}

\begin{framed}
    \begin{proof}
        According to Corollary \ref{thm:continuousoptimal}, solution $\mathbf{k}_\caseTwoA$ is the optimal solution of the continuous relaxation if $M \geq d_0 n$. Since it is integral, it must also be optimal to the integer problem. Thus, we now assume that $M < d_0 n$ and aim to show that any integer feasible solution $\mathbf{k} = (k_0,k_1,k_y,y)$ is dominated by at least one of $\mathbf{k}_
        \caseZeroMinus$, $\mathbf{k}_\caseTwoBMinus$, and $\mathbf{k}_\caseTwoBPlus$. We start by noticing that there are two cases, \eqref{arg:yLessd1} $y \leq d_1$ and \eqref{arg:yBiggerd1} $y \geq d_1$.
        
        \begin{enumerate}[I.]
            \item\label{arg:yLessd1} \big[$y < d_1$\big] By Theorem \ref{thm:k_1isextreme},
                $
                \brObj(k_0,k_1,k_y,y) < \brObj\left(k_0,\left\lfloor M\right\rfloor,k_y^*,y^*\right)
                $
            since $k_0$ is left constant. Notice that $k_1$ has been increased to its maximum integer-feasible value, so $y^* < y \leq d_1 < d_0$. Thus Theorem \ref{thm:y=d_0} tells us that
                $$
                \brObj\left(k_0,\left\lfloor M\right\rfloor,k_y^*,y^*\right) < 
                \brObj\left(\left\lfloor n-M\right\rfloor,\left\lfloor M\right\rfloor,k_y^\circ,y^\circ\right)
                $$
           where $k_y^\circ$ and $y^\circ$ take the values required to achieve equality feasibility; by equations \ref{constr:kspacen} and \ref{constr:kspaceM} these are $k_y^\circ = n-\left\lfloor n-M\right\rfloor-\left\lfloor M\right\rfloor$ and $y^\circ = M-\left\lfloor M\right\rfloor$. Thus, this last solution is exactly $\mathbf{k}_\caseZeroMinus.$
            
            \item\label{arg:yBiggerd1} \big[$y \geq d_1$\big] By Theorem \ref{thm:k_1isextreme},
                $$
                \brObj(k_0,k_1,k_y,y) < \brObj\left(k_0,0,k_y^*,y^*\right)
                $$
            since $k_0$ is left constant. Notice that $k_1$ is decreased, which means that $y$ has increased, so we could have either \eqref{arg:yStarBiggerd0} $y^* \geq d_0$ or \eqref{arg:yStarLess0} $y^* \leq d_0$.
     \begin{enumerate}[i.]
                \item\label{arg:yStarBiggerd0}
                \big[$y^* \geq d_0$\big] In this case, since $k_1 = 0$ \eqref{constr:kspaceM} gives 
                    \begin{align*}
                    M &= y(n-k_0)
                      \geq d_0(n-k_0)
                    \end{align*}
                which implies $k_0 \geq n-\frac{M}{d_0}$. However, because $k_0 \in \mathbb{Z}$ it must be that $k_0 \geq \left\lceil n-\frac{M}{d_0}\right\rceil$. With this, Theorem \ref{thm:y=d_0} tells us that
                    $
                    \brObj\left(k_0,0,k_y^*,y^*\right) \leq \brObj\left(\left\lceil n-\frac{M}{d_0}\right\rceil,0,k_y^\circ,y^\circ\right),
                    $
                which is exactly $\mathbf{k}_\caseTwoBPlus$.
                
                \item\label{arg:yStarLess0} 
                \big[$y^* \leq d_0$\big] In this case, by similar logic as above, $k_0 \leq \left\lfloor n-\frac{M}{d_0}\right\rfloor$. With this, Theorem \ref{thm:y=d_0} tells us that
                    $
                    \brObj\left(k_0,0,k_y^*,y^*\right) < \brObj\left(\left\lfloor n-\frac{M}{d_0}\right\rfloor,0,k_y^\circ,y^\circ\right),
                    $
                which is exactly $\mathbf{k}_\caseTwoBMinus$.
                \end{enumerate}
        \end{enumerate}
        Finally, by comparing the objective value of these candidates, we can find the optimal solution.
    \end{proof}
\end{framed}

Given an optimal solution $\mathbf{k}^* = (k^*_0,k^*_1,k^*_y,y^*)$ of \kspace, any $\mathbf{x} \in \mathcal{X}(k_0^*,k_1^*)$ is an optimal solution of (\ref{model:01space}). 

\section{Conclusion and Future Work}
In this work, we have presented a solution structure to a specific continuous knapsack problem. We have demonstrated its applications in political redistricting and resource allocation problems where expected payoff ratios need to be maximized. Our results provide insight into the approximate structure of solutions to more complicated versions of the problem.

As for future work, one possible direction is to vary the variable ranges for each $x_i$ to non-uniform bounds $[a_i, b_i]$, which can be more realistic in some applications. Another potential avenue for future research is to study the structure of solutions when additional inequalities are present. Our results suggest that many of the variables will attain their lower bounds due to the convex portion of the objective function. This implies that for a more general problem, the optimal solution is likely to be found on the boundary of a low-dimensional face of the feasible region.

Overall, the solution structure presented in this work provides valuable insights into the behavior of solutions for the continuous knapsack problem. The future work can build upon these insights to address more complex and realistic versions of the problem, and to explore its applications in various domains.

\paragraph{Acknowledgements} We would like to thank Nicholas Goedert for inspiration for this paper.

\printbibliography

\end{document}